\documentclass[11pt,twoside]{article}
\usepackage{amsfonts}
\usepackage{mathrsfs}
\usepackage{amssymb,amsmath,amsthm,amsfonts,mathrsfs,hyperref}
\usepackage{times}
\usepackage{enumerate}
\usepackage{cite,titletoc}
\usepackage{titletoc}
\usepackage{epstopdf}
\usepackage{ulem}

\allowdisplaybreaks

\def\cA{\mathcal A}

\def\cD{\mathcal D}

\def\cF{\mathcal F}
\def\cG{\mathcal G}

\def\cJ{\mathcal J}

\def\cM{\mathcal M}

\def\cQ{\mathcal Q}
\def\cR{\mathcal R}
\def\cS{\mathcal S}

\def\cW{\mathcal W}
\def\cX{\mathcal X}

\def\bW{\mathbb W}
\def\sW{\mathscr W}
\def\N{\mathop{\mathbb N\kern 0pt}\nolimits}
\def\Q{\mathop{\mathbb Q\kern 0pt}\nolimits}
\def\R{\mathop{\mathbb R\kern 0pt}\nolimits}
\def\SS{\mathop{\mathbb S\kern 0pt}\nolimits}

\def\ds{\displaystyle}
\def\supp{\mathop{\rm supp}\nolimits}

\def\p{\partial}
\def\f{\frac}
\def\t{\tilde}
\def\dl{\delta}
\def\na{\nabla}
\def\al{\alpha}
\def\eps{\epsilon}
\def\ve{\varepsilon}
\def\vp{\varphi}

\def\dive{\operatorname{div}}
\def\curl{\operatorname{curl}}
\def\supp{\operatorname{supp}}
\def\ls{\lesssim}

\newcommand{\w}[1]{\langle {#1} \rangle}

\hypersetup{colorlinks=true,linkcolor=blue,citecolor=red,urlcolor=cyan}

\theoremstyle{plain}
\newtheorem{theorem}{Theorem}[section]

\newtheorem{lemma}[theorem]{Lemma}
\newtheorem{corollary}[theorem]{Corollary}

\theoremstyle{definition}

\newtheorem{remark}{Remark}[section]

\numberwithin{equation}{section}

\title{Long time existence of smooth solutions to 2D compressible Euler equations of Chaplygin gases with non-zero vorticity}

\author{Fei Hou$^{1, *}$ \qquad Huicheng
  Yin$^{1,2, }$\footnote{Fei Hou (\texttt{fhou$@$nju.edu.cn}) and
    Huicheng Yin (\texttt{huicheng$@$nju.edu.cn}, \texttt{05407$@$njnu.edu.cn}) are supported by
    the NSFC (No.~11731007).}\\
    [12pt] {\small 1. Department of Mathematics, Nanjing University, Nanjing 210093, China}\\
  {\small 2. School of Mathematical Sciences and Mathematical Institute, }\\
  {\small Nanjing Normal University, Nanjing 210023, China}}

\begin{document}

\date{}
\maketitle
\thispagestyle{empty}

\begin{abstract}

For the 2D compressible  isentropic Euler equations of polytropic gases with an initial perturbation of
size $\ve$ of a rest state, it has been known that if the initial data are rotationnally
invariant or irrotational, then the lifespan $T_{\ve}$ of the classical solutions is of order $O(\f{1}{\ve^2})$;
if the initial vorticity is of size $\ve^{1+\al}$ ($0\le\al\le 1$), then $T_{\ve}$ is of $O(\f{1}{\ve^{1+\al}})$.
In the present paper, for the 2D compressible isentropic Euler equations of Chaplygin gases, if
the initial data are a perturbation of size $\ve$, and the initial vorticity is of any size
$\dl$ with  $0<\dl\le \ve$, we will establish  the lifespan $T_{\dl}=O(\f{1}{\dl})$. For examples,
if $\dl=e^{-\f{1}{\ve^2}}$ or $\dl=e^{-e^{\f{1}{\ve^2}}}$ are chosen, then $T_{\dl}=O(e^{\f{1}{\ve^2}})$
or $T_{\dl}=O(e^{e^{\f{1}{\ve^2}}})$ although the perturbations of the initial density and the divergence of the
initial velocity are only of order $O(\ve)$. Our main ingredients are: finding the null condition structures
in 2D  compressible Euler equations of Chaplygin gases and looking for the good unknown;
establishing a new class of  weighted space-time $L^\infty$-$L^\infty$ estimates for the solution itself
and its gradients of 2D linear wave equations; introducing some suitably weighted energies
and taking the $L^p$ $(1<p<\infty)$ estimates on the vorticity.

\vskip 0.2 true cm


\vskip 0.2 true cm
\noindent
\end{abstract}

\vskip 0.6 true cm
\tableofcontents

\section{Introduction}

The 2D compressible isentropic Euler equations are
\begin{equation}\label{Euler}
\left\{
\begin{aligned}
&\p_t\rho+\dive (\rho u)=0\hspace{5.6cm}\text{(Conservation of mass)},\\
&\p_t(\rho u)+\dive (\rho u \otimes u)+\nabla P=0
\hspace{3.3cm}\text{(Conservation of momentum)},\\
\end{aligned}
\right.
\end{equation}
where $(t,x)=(t, x^1, x^2)\in\R^{1+2}_+=[0,\infty)\times\R^2$, $\nabla=(\p_1, \p_2)=(\p_{x^1}, \p_{x^2})$,
and $u=(u_1,u_2),~\rho,~P$ stand for the velocity, density, pressure respectively.
In addition, the pressure $P=P(\rho)$ is a smooth
function of $\rho$ when $\rho>0$, moreover, $P'(\rho)>0$  for $\rho>0$.

For the polytropic gases (see \cite{CF}),
\begin{equation}\label{polytropicGas}
    P(\rho)=A\rho^\gamma,
\end{equation}
where $A$ and $\gamma$ ($1<\gamma<3$) are some positive constants.

For the Chaplygin gases (see \cite{CF} or \cite{Godin07}),
\begin{equation}\label{ChaplyginGas}
    P(\rho)=P_0-\frac{B}{\rho},
\end{equation}
where $P_0>0$ and $B>0$ are constants.

If $(\rho, u)\in C^1$ is a solution of \eqref{Euler} with $\rho>0$,
then \eqref{Euler} is equivalent to the following form
\begin{equation}\label{EulerC1form}
\left\{
\begin{aligned}
&\p_t\rho+\dive (\rho u)=0,\\
&\p_tu+u\cdot\nabla u+\ds\frac{c^2(\rho)}{\rho}\nabla \rho=0,\\
\end{aligned}
\right.
\end{equation}
where the sound speed $c(\rho):=\sqrt{P'(\rho)}$.

Consider the  initial data of \eqref{Euler} as follows
\begin{equation}\label{initial}
(\rho(0,x), u(0,x))=(\bar\rho+\rho^0(x), u^0(x)),
\end{equation}
where $\bar\rho>0$ is a constant, $\bar\rho+\rho^0(x)>0$,
and $\rho^0(x), u^0(x)=(u_1^0(x), u_2^0(x))\in C_0^{\infty}$.  When
\begin{equation}\label{initial:curl}
\curl u^0(x):=\p_1 u_2^0-\p_2 u_1^0\equiv 0,
\end{equation}
as long as $(\rho, u)\in C^1$ for $0\le t\le T_0$, then $\curl u\equiv 0$
always holds for $0\le t\le T_0$. In this case, one can introduce the potential function $\phi$ such that $u=\na \phi$,
then the Bernoulli's law implies $\p_t\phi+\f12|\na\phi|^2+h(\rho)=0$ with $h'(\rho)=\f{c^2(\rho)}{\rho}$ and $h(\bar\rho)=0$.
By the implicit function theorem due to
$h'(\rho)>0$ for $\rho>0$, then the density
function $\rho$ can be expressed as
\begin{equation}\label{density:potential}
\rho=h^{-1}\biggl(-\p_t\phi-\frac{1}{2}|\na\phi|^2\biggr)
=:H(\p\phi),
\end{equation}
where $\p=(\p_t, \na).$ Substituting \eqref{density:potential} into the mass conservation equation in \eqref{Euler}
yields that
\begin{equation}\label{mass:potential}
\p_t(H(\p\phi))+\ds\sum_{i=1}^2\p_i\bigl(H(\p\phi)\p_i\phi\bigr)=0.
\end{equation}
For any $C^2$ solution $\phi$, \eqref{mass:potential} can be
rewritten as the following second order quasilinear equation
\begin{equation}\label{QLW}
\p_t^2\phi+2\ds\sum_{k=1}^2\p_k\phi\p_{tk}^2\phi-c^2(\rho)\Delta\phi
+\ds\sum_{i,j=1}^{2}\p_{i}\phi\p_j\phi\p_{ij}^2\phi
=0,
\end{equation}
where $c(\rho)=c(H(\p\phi))$, and the Laplace operator
$\Delta:=\ds\sum_{i=1}^{2}\p_i^2$. Without loss of generality and for simplicity,
$c(\bar\rho)=1$ can be supposed, and then $c^2(\rho)=1-2\bar\rho c'(\bar\rho)\p_t\phi+O(|\p\phi|^2)$.
Especially, in the case of the Chaplygin gases, \eqref{QLW} is
\begin{equation}\label{QLW:Chaplygin}
\p_t^2\phi-\triangle\phi+2\sum_{k=1}^2\p_k\phi\p_t\p_k\phi
-2\p_t\phi\triangle\phi+\sum_{i,j=1}^2\p_i\phi\p_j\phi\p_{ij}^2\phi
-|\nabla\phi|^2\triangle\phi=0.
\end{equation}

When $\|\rho^0(x)\|_{H^4}+\|u^0(x)\|_{H^4}\le \bar\ve$ and $\bar\ve>0$ is sufficiently small, if follows from
Theorem 6.5.3 of \cite{Hormander97book} and equation \eqref{QLW} that the lifespan $T_{\bar\ve}$ of smooth solution $(\rho, u)$
to \eqref{EulerC1form} fulfills $T_{\bar\ve}\ge\frac{C}{\bar\ve^2}$, where $C>0$ is a constant depending only on the
initial data. In addition, for the polytropic gases, since the first null condition does not hold
for equation \eqref{QLW}, then $T_{\bar\ve}\le\frac{\t C}{\bar\ve^2}$ holds for suitably positive constant $\t C$
(see \cite{Alinhac99} and \cite{John90}); for the Chaplygin gases, note that both the first null condition and
the second null condition hold
for equation \eqref{QLW:Chaplygin}, then $T_{\bar\ve}=+\infty$ holds (see \cite{Alinhac01}).

When
\begin{equation}\label{curl0:nonvanish}
\curl u^0(x)\not\equiv 0,
\end{equation}
if
\begin{equation*}
\curl u^0(x)=O(\bar\ve^{1+\alpha}),
\end{equation*}
where $\alpha\ge 0$ is a constant, then by Theorem 1 and Theorem 2 of \cite{Sideris97} that
the lifespan $T_{\bar\ve}$ of smooth solution $(\rho, u)$
to \eqref{EulerC1form} satisfies
\begin{equation}\label{curl:lifespan}
T_{\bar\ve}\ge\frac{C}{\bar\ve^{\min\{1+\alpha,2\}}}.
\end{equation}
Note that $T_{\bar\ve}$ in \eqref{curl:lifespan} is also optimal for the polytropic gases, i.e.,
$T_{\bar\ve}\le\frac{\bar C}{\bar\ve^{\min\{1+\alpha,2\}}}$ holds for some constant $\bar C>C$, one can see
\cite{Alinhac92}-\cite{Alinhac93}, \cite{John90} and \cite{Sideris97}. With respect to more results on the blowup or the blowup mechanism of \eqref{Euler}  for polytropic gases,
the papers \cite{Christodoulou07}-\cite{CM14},  \cite{Godin05}, \cite{HKSW16},
\cite{LukSpeck18}, \cite{Rammaha89}, \cite{Sideris85}, \cite{Yin} can be referred.
In the present paper, we intend to study the long time existence of smooth solutions to \eqref{EulerC1form} for the Chaplygin gases
with initial data \eqref{initial} and
\begin{align}\label{curl0:order}
\curl u^0(x)\not\equiv 0, \quad \curl u^0(x)=O(\delta),
\end{align}
where $\delta>0$ and $\delta=o(\ve)$. For this end, we introduce the following quantities on the initial data and vorticity
\begin{equation}\label{initial:data}
\begin{split}
&\ve=\sum_{k\le N}\|(\w{|x|}\nabla)^k(u^0(x),\frac{\rho^0(x)}{\rho^0(x)+\bar\rho})\|_{L_x^2}
+\sum_{k\le N_1-1}\|\w{|x|}^2(\w{|x|}\nabla)^k\Delta^{-1}\dive u^0(x)\|_{L_x^1}\\
&\qquad+\sum_{k\le N_1-2}\|\w{|x|}^2(\w{|x|}\nabla)^k(\Delta^{-1}\curl(u^0(x)\curl u^0(x)),u^0(x),\frac{\rho^0(x)}{\rho^0(x)+\bar\rho})\|_{L_x^1},
\end{split}
\end{equation}
and
\begin{equation}\label{delta:def}
\delta=\sum_{k\le N-1}\|\w{|x|}(\w{|x|}\nabla)^k\curl u^0(x)\|_{L_x^2}
+\sum_{k\le N_1}\sum_{p=\frac{10}{9},\frac{10}{7},5}
\|\w{|x|}^8(\w{|x|}\nabla)^k\curl u^0(x)\|_{L_x^p},
\end{equation}
where $\w{|x|}=\sqrt{1+|x|^2}$, the integers $N_1$ and $N$ fulfill $N_1\ge13$ and $N_1+2\le N\le2N_1-11$. In addition,
the state equation in \eqref{ChaplyginGas} is conveniently written as
\begin{equation}\label{Chaplygin:gas}
P(\rho)=P_0-\frac{\bar\rho^2}{\rho}.
\end{equation}

\begin{theorem}\label{thm:2dChaplygin}
There exists three constants $\ve_0,\delta_0,\kappa>0$ such that when the initial data $(\rho^0,u^0)$ satisfies $\ve\le\ve_0$ and $\delta\le\delta_0$,
then \eqref{Euler} with \eqref{Chaplygin:gas} admits a solution $(\rho-\bar\rho,u)\in C([0,T_\delta];H^N(\R^2))$ with $T_\delta=\frac{\kappa}{\delta}$.
\end{theorem}

\begin{remark}
In Theorem \ref{thm:2dChaplygin}, when $\delta=\ve^{1+\alpha}$ with $0\le\alpha\le1$ and $0<\ve\le\ve_0$, then $T_\delta=\frac{\kappa}{\delta}$
has been shown in Theorem 2 of \cite{Sideris97}. Hence we can assume $\dl\le \ve^\frac87$ in Theorem \ref{thm:2dChaplygin}
without loss of generality.
\end{remark}

\begin{remark}
If we choose $\delta=\ve^\ell$ with $\ell\ge2$ or $e^{-\frac{1}{\ve^p}}$ with $p>0$ in Theorem \ref{thm:2dChaplygin},
then the existence time of smooth solution $(\rho,u)$ to \eqref{Euler} for the Chaplygin gases  is larger
than $\frac{\kappa}{\ve^\ell}$ or $\kappa e^{\frac{1}{\ve^p}}$.
\end{remark}

\begin{remark}
If $(\rho^0(x), u^0(x))\in C_0^{\infty}(B(0, R))$, where $B(0,R)$ is a ball with the center at
the origin and the radius $R>0$, then $\|\w{|x|}^2(\w{|x|}\nabla)^k(\Delta^{-1}\curl(u^0(x)\curl u^0(x)),u^0(x),\frac{\rho^0(x)}{\rho^0(x)+\bar\rho})\|_{L_x^1}$
and $\|\w{|x|}^2(\w{|x|}\nabla)^k\Delta^{-1}\dive u^0(x)\|_{L_x^1}$
in \eqref{initial:data} can be replaced by
\newline
$\|\nabla^k(\Delta^{-1}\curl(u^0(x)\curl u^0(x)),u^0(x),\frac{\rho^0(x)}{\rho^0(x)+\bar\rho})\|_{L_x^1(B(0,R))}$
and
$\|\nabla^k\Delta^{-1}\dive u^0(x)\|_{L_x^1(B(0,R))}$
respectively. The reasons are:

{\bf $\bullet$} from the Helmholtz decomposition of initial velocity $u^0$, one has $u^0=P_1u^0+P_2u^0$,
where $P_1u^0=-\nabla(-\Delta)^{-1}\dive u^0$, $P_2u^0=-\nabla^\perp(-\Delta)^{-1}\curl u^0$ and $\nabla^\perp=(-\p_{x^2},\p_{x^1})$.
When $\supp u^0\subset B(0,R)$, we choose a smooth cut-off function $\eta(x)$ such that
$\eta|_{\supp u^0}=1$ and  $\supp\eta\subset B(0,R)$, and then $u^0=\eta P_1u^0+\eta P_2u^0$.
In this case, the related term $\Delta^{-1}\dive u^0(x)$ can be thought to be supported  in $B(0,R)$.

{\bf $\bullet$} from \eqref{dt:phi} in Appendix A, we have $\p_t\phi^0+\frac12|u^0|^2+\sigma^0-\frac12(\sigma^0)^2
=-(-\Delta)^{-1}\curl(u^0\curl u^0)$, where $P_1u^0=\nabla \phi^0$ and $\sigma^0=\frac{\rho^0}{\rho^0+\bar\rho}$.
When $\supp u^0, \supp\rho^0\subset B(0,R)$, then $\supp\phi^0,\supp\sigma^0\subset B(0,R)$.
As in the above,  $\Delta^{-1}\curl(u^0\curl u^0)$ can be thought to be supported  in $B(0,R)$.

\end{remark}

\begin{remark}
Consider the 2D full compressible Euler equations of Chaplygin gases
\begin{equation}\label{full:Euler}
\left\{
\begin{aligned}
&\p_t\rho+\dive(\rho u)=0\hspace{5.6cm}\text{(Conservation of mass)},\\
&\p_t(\rho u)+\dive(\rho u \otimes u)+\nabla P=0
\hspace{3.3cm}\text{(Conservation of momentum)},\\
&\p_t(\rho e+\frac12\rho|u|^2)+\dive((\rho e+\frac12\rho|u|^2+P)u)=0
\qquad\text{(Conservation of energy)},\\
&\rho(0,x)=\bar\rho+\ve\rho^0(x), u(0,x)=\ve u^0(x), S(0,x)=\bar S+\ve S^0(x),
\end{aligned}
\right.
\end{equation}
where $P=P(\rho,S)$,  $e=e(\rho,S)$, $S$ stand for the pressure, inner energy and entropy respectively.
In addition, $\ve>0$ is small, and $(\rho^0(x),u^0(x),S^0(x))\in C_0^{\infty}(\Bbb R^2)$. If $\rho^0(x)=\rho^0(r),~S^0(x)=S^0(r)$ and $u^0(x)=f^0(r)\frac{x}{r}+g^0(r)\frac{x^\perp}{r}$ with $x^\perp=(-x^2,x^1)$ and $r=|x|$, then we have shown \eqref{full:Euler} has a global smooth solution $(\rho, u, S)$ in \cite{HouYin19}
and \cite{HouYin20} (when $g^0(r)\equiv 0$, the global existence of smooth symmetric solutions $(\rho, u, S)$ to 2D and 3D systems
\eqref{full:Euler} has been established in \cite{DWY15} and \cite{Godin07} respectively). By combining the methods in the paper and \cite{HouYin20}, we can actually establish that for the perturbed problem
of \eqref{full:Euler}
\begin{equation}\label{full:Euler:1}
\left\{
\begin{aligned}
&\p_t\rho+\dive(\rho u)=0,\\
&\p_t(\rho u)+\dive(\rho u \otimes u)+\nabla p=0,\\
&\p_t(\rho e+\frac12\rho|u|^2)+\dive((\rho e+\frac12\rho|u|^2+P)u)=0,\\
&\rho(0,x)=\bar\rho+\ve\rho^0(x)+\delta\rho^1(x), u(0,x)=\ve u^0(x)+\delta u^1(x),\\
&S(0,x)=\bar S+\ve S^0(x)+\delta S^1(x),
\end{aligned}
\right.
\end{equation}
where $\curl u^0(x)\equiv0$ (or $\rho^0(x)=\rho^0(r),~S^0(x)=S^0(r)$ and $u^0(x)=f^0(r)\frac{x}{r}+g^0(r)\frac{x^\perp}{r}$),
$\delta=o(\ve)$, and $(\rho^1(x),u^1(x),S^1(x))\in C_0^{\infty}(\Bbb R^2)$, then \eqref{full:Euler:1}
admits a smooth solution $(\rho,u, S)$ for $t\in [0,T_\delta]$ with $T_\delta=\frac{\kappa}{\delta}$ as in Theorem \ref{thm:2dChaplygin}.
\end{remark}

\begin{remark}
We mention the interesting works on the Euler-Maxwell system, which are related our result.
The global smooth, small-amplitude, irrotational solution to the Euler-Maxwell two-fluid system was proved in \cite{GIP16}.
Considering the influence of the vorticity, Ionescu and Lie in \cite{IL18} have shown that the existence time is
larger than $\frac{C}{\delta}$, where $\delta$ is the size of initial vorticity and $C$ is some positive constant.
\end{remark}

\begin{remark}
It is well known that the compressible Euler system is symmetric hyperbolic with respect to the
time $t$ when the vacuum does not appear. A.~Majda posed the following conjecture on Page 89 of \cite{Majda84book}:

{\bf Conjecture.} {\it If the multidimensional nonlinear symmetric system is totally linearly degenerate, then it typically has
smooth global solutions when the initial data are in $H^s(\R^n)$ with $s>\frac{n}{2}+1$ unless the solution itself blows up in finite time.}

Note that the compressible Euler system of Chaplygin gases is  totally linearly degenerate,
A. Majda's conjecture together with the opinion in \cite{Alinhac95} can yield the following open question:

{\bf  Open question.} {\it For the $n-$dimensional ($n\ge 2$) full compressible Euler equations of Chaplygin gases
\begin{equation*}\label{full:Euler:2}
\left\{
\begin{aligned}
&\p_t\rho+\dive(\rho u)=0,\\
&\p_t(\rho u)+\dive(\rho u \otimes u)+\nabla P=0,\\
&\p_t(\rho e+\frac12\rho|u|^2)+\dive((\rho e+\frac12\rho|u|^2+P)u)=0,\\
&\rho(0,x)=\bar\rho+\ve\rho^0(x), u(0,x)=\ve u^0(x), S(0,x)=\bar S+\ve S^0(x),
\end{aligned}
\right.
\end{equation*}
where $\ve>0$ is small, and $(\rho^0(x),u^0(x),S^0(x))\in C_0^{\infty}(\Bbb R^2)$
with $u^0(x)=(u_1^0(x), \cdot\cdot\cdot, u_n^0(x))$, then $T_{\ve}=+\infty$.}

Although this open question has not been solved so far, our result in Theorem 1.1 illustrates that
the order of lifespan $T_{\ve}$ is only essentially influenced by the initial vorticity.
\end{remark}

We next give some comments on the proof of Theorem~\ref{thm:2dChaplygin}. From now on,
$\bar\rho=1$ is always assumed.
Introducing the perturbed sound speed $\sigma(t,x)=1-\frac{1}{\rho(t,x)}$ as the new unknown,
then \eqref{Euler} is reduced to
\begin{equation}\label{reducedEuler}
\left\{
\begin{aligned}
&\p_t\sigma+\dive u=Q_1:=\sigma\dive u-u\cdot\nabla\sigma,\\
&\p_tu+\nabla\sigma=Q_2:=\sigma\nabla\sigma-u\cdot\nabla u,
\end{aligned}
\right.
\end{equation}
where the $i$-th component of the vector $Q_2$ is $Q_{2i}=\sigma\p_i\sigma-u\cdot\nabla u_i$
($i=1,2$). In addition, we also define the good unknown $g$ in the region $|x|>0$ as follows
\begin{equation}\label{goodunknown:def}
g:=(g_1,g_2)=u-\omega\sigma\quad\text{with $g_i=u_i-\omega_i\sigma$, $i=1,2$,}
\end{equation}
where $\omega=(\omega_1,\omega_2):=(\frac{x^1}{|x|},\frac{x^2}{|x|})\in\SS^1$. We point out that the introduction of
$g$ is motivated by the second order wave equations although  \eqref{Euler} admits the non-zero and small higher
order  vorticity: by the potential equation \eqref{QLW:Chaplygin},
then $u_i=\p_i\phi$ and $\sigma=-\p_t\phi+\text{higher order error terms of $\p\phi$}$,
which derives  $g_i=(\p_i+\omega_i\p_t)\phi+\text{higher order error terms of $\p\phi$}$.
It is well known that $(\p_i+\omega_i\p_t)\phi$ is the good derivative in the study of the nonlinear wave
equation \eqref{QLW:Chaplygin} (see \cite{Alinhac10}) since $(\p_i+\omega_i\p_t)\phi$ will admit more rapid
space-time decay rates. By some ideas and methods dealing with the null condition structures
in \cite{HouYin19,HouYin20,HouYin20jde}, we obtain the better $L^\infty$ space-time decay rates of $g$.
Based on this, the elementary energy $E_N(t)$ can be estimated (see Lemma ~\ref{lem:energy:wave} in Section~\ref{sect5:1}).
Nevertheless, we have to overcome other essential difficulties which are arisen by the interaction between the irrotational
part of the velocity and the vorticity. To solve the resulting difficulties, our ingredients are:

{\bf $\bullet$} By the transport equation $\ds(\p_t+u\cdot\nabla)\big(\frac{\curl u}{\rho}\big)=0$
and the careful analysis, we can derive that the key influence of the vorticity is
concentrated in the interior of the outgoing light cone.

{\bf $\bullet$} Near the  outgoing light conic surface, our first observation  is that the system \eqref{Euler}
can be changed into the second order potential flow equation. However, the optimal time-decay rate
of solutions to the 2D free wave equation is merely $(1+t)^{-\f12}$,  which is far to derive the existence time $\ds T_\delta
=\f{\kappa}{\dl}$ in Theorem  \ref{thm:2dChaplygin}. The reason is due to: for example,  when $\dl=e^{-e^{\frac{1}{\ve^2}}}$ is chosen,
then the integral $\ds\int_0^{T_{\dl}}\frac{dt}{\sqrt{1+t}}=O(e^{\f{1}{2}e^{\frac{1}{\ve^2}}})$ is sufficiently large  as $\ve\rightarrow0$,
which leads to that the related energy $E(t)$ can not be controlled well by the corresponding energy
inequality $E(t)\le E(0)+\f{C\ve}{\sqrt{1+t}}E(t)$. To overcome this difficulty, our second observation is that
the velocity is the gradient of the potential and the potential satisfies a second order quasilinear wave equation
with the first and the second null conditions. By establishing a new type of weighted $L^\infty$-$L^\infty$ estimate
for the derivatives of the potential, the better space-time decay rate of $u$ can be obtained (see Corollary \ref{coro:improve} in
Section \ref{sect4}).

Based on the key estimates in the above, we eventually get the $L^2$ and other $L^p$
(for some suitable numbers $p$ with $p\not=2$) energy estimates of the vorticity
and further close the bootstrap assumptions in Section \ref{sect2}.

This paper is organized as follows.
In Section~\ref{sect2}, we will introduce the basic bootstrap assumptions, Helmholtz decomposition and
some elementary pointwise estimates.
The estimates of the good unknown $g$ and some auxiliary energies are derived in Section~\ref{sect3}.
In Section~\ref{sect4}, by establishing a new type of the weighted $L^\infty$-$L^\infty$ for the 2D wave equations,
the required  pointwise estimates of the solution $(\sigma, u)$ with suitable space-time decay rates are derived.
Based on the pointwise estimates in Section~\ref{sect4}, the Hardy inequality and the ghost weight method in \cite{Alinhac01},
we get the related energy estimates in Section~\ref{sect5}.
In Section~\ref{sect6}, with the previous energy inequalities and Gronwall's inequalities,
the proof of Theorem~\ref{thm:2dChaplygin} is finished.

\section{Some preliminaries}\label{sect2}

\subsection{The vector field and bootstrap assumptions}
Define the spatial rotation vector field
\begin{equation*}
\Omega:=x_1\p_2-x_2\p_1.
\end{equation*}
For the vector-valued function $U=(U_1,U_2)$, denote
\begin{equation*}
\tilde\Omega U:=\Omega U-U^\perp=(\Omega U_1+U_2,\Omega U_2-U_1).
\end{equation*}
Let $\tilde\Omega U_k=(\tilde\Omega U)_k$ be the $k$-component of $\tilde\Omega U$ rather than
the operator $\tilde\Omega$ act on the component $U_k$.

According to the definitions of $\Omega$ and $\tilde\Omega$, it is easy to check that for the scalar function $f$ and
the vector-valued functions $U,V$,
\begin{equation}\label{rotation:commutation}
\begin{array}{ll}
  \Omega\dive U=\dive\tilde\Omega U,\qquad
  & \Omega\curl U=\curl\tilde\Omega U, \\
  \tilde\Omega\nabla f=\nabla\Omega f,\qquad
  & \tilde\Omega(U\cdot\nabla V)=U\cdot\nabla(\tilde\Omega V)
  +(\tilde\Omega U)\cdot\nabla V,\\
  \Omega(U\cdot V)=(\tilde\Omega U)\cdot V+U\cdot\tilde\Omega V.&
\end{array}
\end{equation}
The spatial derivatives can be decomposed into the radial and angular components for $r=|x|\not=0$:
\begin{equation*}
\nabla=\omega\p_r+\frac{\omega^\perp}{|x|}\Omega,
\end{equation*}
where $\omega^\perp:=(-\omega_2,\omega_1)$.
For convenience and simplicity, we often denote this decomposition as
\begin{equation}\label{radial:angular}
\p_i=\omega_i\p_r+\frac{1}{|x|}\Omega.
\end{equation}
For the multi-index $a$, set
\begin{equation}\label{vectorfield}
\cS:=t\p_t+r\p_r,\quad \Gamma^a=\cS^{a_s}Z^{a_z},\quad Z\in\{\p,\Omega\},
\qquad \tilde\Gamma^a=\cS^{a_s}\tilde Z^{a_z},
\quad \tilde Z\in\{\p,\tilde\Omega\}.
\end{equation}
By acting $(\cS+1)^{a_s}Z^{a_z}$ and $(\cS+1)^{a_s}\tilde Z^{a_z}$ on the equations in \eqref{reducedEuler}, respectively,
we then have
\begin{equation}\label{high:eqn}
\begin{split}
\left\{
\begin{aligned}
&\p_t\Gamma^a\sigma+\dive\tilde\Gamma^au=\cQ_1^a:=\sum_{b+c=a}C^a_{bc}Q_1^{bc},\\
&\p_t\tilde\Gamma^au+\nabla\Gamma^a\sigma=\cQ_2^a:=\sum_{b+c=a}C^a_{bc}Q_2^{bc},
\end{aligned}
\right.
\end{split}
\end{equation}
where $C^a_{bc}$ are constants ($C^a_{a0}=C^a_{0a}=1$) and
\begin{equation}\label{Qbc:def}
\begin{split}
Q_1^{bc}:=&\Gamma^b\sigma\dive\tilde\Gamma^cu-\tilde\Gamma^bu\cdot\nabla\Gamma^c\sigma,\\
Q_2^{bc}:=&\Gamma^b\sigma\nabla\Gamma^c\sigma-\tilde\Gamma^bu\cdot\nabla\tilde\Gamma^cu.
\end{split}
\end{equation}
It is convenient to introduce the specific vorticity
\begin{equation}\label{curl:def}
w:=\frac{\curl u}{\rho}=(1-\sigma)\curl u
\end{equation}
since $(\p_t+u\cdot\nabla)w=0$ holds.
For integer $m\in\N$, we define
\begin{equation}\label{energy:def}
\begin{split}
E_m(t):=&~\sum_{|a|\le m}\|(\tilde\Gamma^au,\Gamma^a\sigma)(t,x)\|_{L_x^2},\\
\cX_m(t):=&~\sum_{|a|\le m-1}\|\w{|x|-t}(\p_t\tilde\Gamma^au,\dive\tilde\Gamma^au,
\nabla\Gamma^a\sigma,\p_t\Gamma^a\sigma)(t,x)\|_{L_x^2},\\
W_m(t):=&~\sum_{|a|\le m}\|\w{|x|}\Gamma^aw(t,x)\|_{L_x^2},\\
\cW_m(t):=&~\sum_{|a|\le m}\|\w{|x|}\Gamma^a\curl u(t,x)\|_{L_x^2},\\
\bW_m(t):=&~\sum_{|a|\le m}\Big\{\|\w{|x|}^8\Gamma^a\curl u(t,x)\|_{L_x^5}
+\sum_{p=\frac{10}{9},\frac{10}{7}}\|\w{|x|}\Gamma^a\curl u(t,x)\|_{L_x^p}\Big\},\\
\sW_m(t):=&~\sum_{|a|\le m}\Big\{\|\w{|x|}^8\Gamma^aw(t,x)\|_{L_x^5}
+\sum_{p=\frac{10}{9},\frac{10}{7}}\|\w{|x|}\Gamma^aw(t,x)\|_{L_x^p}\Big\}
\end{split}
\end{equation}
and $\cX_0(t)=0$.

Throughout the whole paper, we will make the following bootstrap assumptions:
\begin{equation}\label{bootstrap}
\begin{split}
&{\it for}~~t\delta\le\kappa,\quad\left\{
\begin{aligned}
&E_N(t)+\cX_N(t)\le M\ve(1+t)^{M'\ve},\\
& E_{N_1-4}(t)+\cX_{N_1-4}(t)\le M\ve,\\
&W_{N-1}(t)+\cW_{N-1}(t)+\bW_{N_1}(t)+\sW_{N_1}(t)\le M\delta(1+t)^{M'\ve},\\
&W_{N_1-4}(t)+\cW_{N_1-4}(t)+\bW_{N_1-4}(t)+\sW_{N_1-4}(t)\le M\delta,
\end{aligned}
\right.\\
&\delta\le\ve^\frac87,\qquad M(\ve+\kappa)\le1, \qquad M\ge1, \qquad M'>0,
\end{split}
\end{equation}
where the constant  $M\ge 1$ will be chosen, $M'>0$ is some fixed constant. In Section~\ref{sect6}, we will prove that the constant $M$ on the
right hands of the first four lines in \eqref{bootstrap}
can be improved to $\frac34 M$.

\subsection{The Helmholtz decomposition and commutator}\label{sect:decomp}
For the rapidly decaying vector $U=(U_1,U_2)$ with respect to the space variable $x$,
we divide it into the curl-free part $P_1U$ (irrotational) and the divergence-free part $P_2U$ (solenoidal),
which is called the Helmholtz decomposition
\begin{equation}\label{Helmholtz}
\begin{split}
&P_1U=\nabla\Phi,P_2U=\nabla^\perp\Psi,\Delta\Phi=\dive U,\Delta\Psi=\curl U,\\
&U=P_1U+P_2U=-\nabla(-\Delta)^{-1}\dive U-\nabla^\perp(-\Delta)^{-1}\curl U,
\end{split}\end{equation}
where $\nabla^\perp=(-\p_{x^2},\p_{x^1})$ and $\Delta$ is the Laplacian operator.
It is easy to know
\begin{equation}\label{Helmholtz:id}
\|U\|^2_{L^2}=\|P_1U\|^2_{L^2}+\|P_2U\|^2_{L^2}.
\end{equation}

Next, we give some divergence-curl inequalities.
\begin{lemma}\label{lem:div:curl}
For any vector function $U$, $1<p<\infty$ and $0\le\beta<2(p-1)$, it holds that
\begin{equation}\label{div:curl:ineq}
\begin{split}
\|\w{|x|}^\beta\nabla U\|_{L^p}&\ls\|\w{|x|}^\beta\dive U\|_{L^p}
+\|\w{|x|}^\beta\curl U\|_{L^p},\\
\|\w{|x|-t}\nabla U\|_{L^2}&\ls\|U\|_{L^2}+\|\w{|x|-t}\dive U\|_{L^2}
+\|\w{|x|-t}\curl U\|_{L^2}.
\end{split}
\end{equation}
\end{lemma}
\begin{proof}
From the second line of \eqref{Helmholtz}, we get
\begin{equation*}
\nabla U=-\nabla^2(-\Delta)^{-1}\dive U-\nabla\nabla^\perp(-\Delta)^{-1}\curl U,
\end{equation*}
where $\nabla^2(-\Delta)^{-1}$ and $\nabla\nabla^\perp(-\Delta)^{-1}$ are the bounded operators from $L^p$
to $L^p$ ($1<p<\infty$).
On the other hand, $\w{|x|}^\beta$ belongs to $A_p$ class with $0\le\beta<2(p-1)$
(see \cite{Stein}). Therefore, the first inequality in \eqref{div:curl:ineq} is derived.

Since the second inequality in \eqref{div:curl:ineq} follows from the direct integration by parts,
we omit the details here.
\end{proof}

\begin{lemma}[Commutator]\label{lem:commutation}
For the vector fields $\tilde\Gamma$ defined in \eqref{vectorfield}, one has $[\tilde\Gamma,P_1]:=\tilde\Gamma P_1-P_1\tilde\Gamma=0$, $[\tilde\Gamma,P_2]=0$.
\end{lemma}
\begin{proof}
We only prove $\tilde\Gamma P_1U=P_1(\tilde\Gamma U)$, since it always holds that $[\tilde\Gamma,P_2]=[\tilde\Gamma,{\rm Id}-P_1]=0$.
Note that $[\tilde\Gamma,P_1]=0$ is obvious for $\tilde\Gamma\in\{\p\}$, we now  focus on the case
of $\tilde\Gamma\in\{\tilde\Omega,\cS\}$.

According to \eqref{rotation:commutation} and the first line of \eqref{Helmholtz}, one has
\begin{equation*}
-\Delta P_1(\tilde\Omega U)=-\nabla\dive\tilde\Omega U=-\tilde\Omega\nabla\dive U
=\tilde\Omega(-\Delta)P_1U=-\Delta(\tilde\Omega P_1U).
\end{equation*}
This derives $P_1(\tilde\Omega U)=\tilde\Omega P_1U$ by the uniqueness of the solution to $\Delta$
since $U$ fulfills the rapid decay assumption and $P_1(\tilde\Omega U), \tilde\Omega P_1U$ decay for the space variable.

Analogously, $P_1(\cS U)=\cS P_1U$ follows from
\begin{equation*}
-\Delta P_1(\cS U)=-\nabla\dive\cS U=-(\cS+2)\nabla\dive U
=(\cS+2)(-\Delta)P_1U=-\Delta(\cS P_1U).
\end{equation*}
\end{proof}

\subsection{The elementary pointwise estimates}\label{sect2:3}

\begin{lemma}\label{lem:pw:prepare}
Let $f(t,x)$ be a scalar function, then it holds that
\begin{align}
&\w{|x|}^\frac1p|f(t,x)|\ls\sum_{j=0}^1\sum_{|a|=0}^{2-j}
\|\nabla^a\Omega^jf(t,y)\|_{L_y^p},\quad 1<p<\infty,\label{pw:prepare1}\\
&\w{|x|}^\frac12\w{|x|-t}|f(t,x)|\ls\sum_{j=0}^1\sum_{|a|=0}^{2-j}
\|\w{|y|-t}\nabla^a\Omega^jf(t,y)\|_{L_y^2},\label{pw:prepare2}\\
&\w{|x|}^\frac12\w{|x|-t}^\frac12|f(t,x)|
\ls\sum_{j=0}^1\Big\{\|\Omega^jf(t,y)\|_{L_y^2}
+\sum_{|a|=1}^{2-j}\|\w{|y|-t}\nabla^a\Omega^jf(t,y)\|_{L_y^2}\Big\},\label{pw:prepare3}\\
&\|f(t,x)\|_{L_x^\infty}\ls\ln^\frac12(2+t)\|\nabla f(t,y)\|_{L_y^2}
+\w{t}^{-1}(\|f(t,y)\|_{L_y^2}+\|\nabla^2f(t,y)\|_{L_y^2}),\label{pw:prepare4}\\
&\w{|x|}^\frac83|f(t,x)|\ls\|\w{|y|}^7\nabla\Omega^{\le1}f(t,y)\|_{L_y^5}
+\sum_{p=\frac{10}{9},\frac{10}{7}}
\|\nabla\Omega^{\le1}f(t,y)\|_{L_y^p},\label{pw:prepare5}\\
&\w{|x|}|f(t,x)|\ls\|f(t,y)\|_{H_y^2}+\|\Omega^{\le1}f(t,y)\|_{L_y^2}
+\|\w{y}\nabla\Omega^{\le1}f(t,y)\|_{L_y^2},\label{pw:prepare6}
\end{align}
where $\Omega^{\le1}f$ stands for $\sum_{j\le 1}\Omega^jf$.
\end{lemma}
\begin{proof}
The inequality \eqref{pw:prepare1} with $p=2$ is just (3.1) of \cite{Sideris97}.
We now deal with the general case of $p\in(1,\infty)$. Note that \eqref{pw:prepare1}  in the region $|x|\le1$
follows from the Sobolev embedding $W^{2,p}(\R^2)\hookrightarrow L^\infty(\R^2)$.

For $|x|\ge1$, by the Sobolev embedding on the unit circle $W^{1,p}(\SS^1)\hookrightarrow L^\infty(\SS^1)$
and the Newton-Leibniz formula in the radial direction, we arrive at
\begin{equation*}
\begin{split}
|x||f(t,x)|^p=|x||f(t,|x|\omega)|^p&\ls|x|\int_{\SS^1}|\Omega^{\le1}f(t,|x|\omega)|^pd\omega\\
&\ls|x|\int_{|x|}^\infty\int_{\SS^1}|\Omega^{\le1}f(t,r\omega)|^{p-1}
|\p_r\Omega^{\le1}f(t,r\omega)|d\omega dr\\
&\ls\|\Omega^{\le1}f(t,y)\|^p_{L_y^p}+\|\nabla\Omega^{\le1}f(t,y)\|^p_{L_y^p},
\end{split}
\end{equation*}
which yields \eqref{pw:prepare1}.

For the inequalities \eqref{pw:prepare2}--\eqref{pw:prepare4}, see (3.2), (3.4) of \cite{Sideris97} and (3.4) of \cite{Lei16}, respectively.

Next, we start to prove \eqref{pw:prepare5}.
By the Sobolev embedding $W^{1,5}(\R^2)\hookrightarrow L^\infty(\R^2)$ and $\dot W^{1,\frac{10}{7} }(\R^2)\hookrightarrow L^5(\R^2)$,
we have
\begin{equation*}
\begin{split}
|f(t,x)|&\ls\|f(t,y)\|_{L_y^5}+\|\nabla f(t,y)\|_{L_y^5}\\
&\ls\|\nabla f(t,y)\|_{L_y^\frac{10}{7}}+\|\nabla f(t,y)\|_{L_y^5},
\end{split}
\end{equation*}
this implies  \eqref{pw:prepare5} when $|x|\le1$.

For $|x|\ge1$, it follows from the Sobolev embedding $W^{1,3}(\SS^1)\hookrightarrow L^\infty(\SS^1)$ that
\begin{equation*}
\begin{split}
|f(t,x)|^3=|f(t,|x|\omega)|^3&\ls\int_{\SS^1}|\Omega^{\le1}f(t,|x|\omega)|^3d\omega\\
&\ls\int_{|x|}^\infty\int_{\SS^1}|\Omega^{\le1}f(t,r\omega)|^2
|\p_r\Omega^{\le1}f(t,r\omega)|d\omega dr.
\end{split}
\end{equation*}
Multiplying this inequality by $|x|^8$ and then applying the H\"{o}lder inequality infer
\begin{equation*}
\begin{split}
|x|^8|f(t,x)|^3&\ls\|\w{|y|}^7\nabla\Omega^{\le1}f(t,y)\|_{L_y^5}
\Big\|\big|\Omega^{\le1}f(t,y)\big|^2\Big\|_{L_y^\frac54}\\
&\ls\|\w{|y|}^7\nabla\Omega^{\le1}f(t,y)\|_{L_y^5}\|\Omega^{\le1}f(t,y)\|^2_{L_y^\frac52}\\
&\ls\|\w{|y|}^7\nabla\Omega^{\le1}f(t,y)\|_{L_y^5}
\|\nabla\Omega^{\le1}f(t,y)\|^2_{L_y^\frac{10}{9}},
\end{split}
\end{equation*}
where we have used the Sobolev embedding $\dot W^{1,\frac{10}{9}}(\R^2)\hookrightarrow L^\frac52(\R^2)$.
Therefore, we achieve \eqref{pw:prepare5}.

Finally, we turn to the proof of \eqref{pw:prepare6}.
For the case of $|x|\le1$, \eqref{pw:prepare6} is a direct result of the Sobolev embedding $H^2(\R^2)\hookrightarrow L^\infty(\R^2)$.

For $|x|\ge1$, similarly to the proof of \eqref{pw:prepare5}, applying $W^{1,2}(\SS^1)\hookrightarrow L^\infty(\SS^1)$ instead leads to
\begin{equation*}
\begin{split}
|x|^2|f(t,x)|^2=|x|^2|f(t,|x|\omega)|^2&\ls|x|^2\int_{\SS^1}|\Omega^{\le1}f(t,|x|\omega)|^2d\omega\\
&\ls|x|^2\int_{|x|}^\infty\int_{\SS^1}|\Omega^{\le1}f(t,r\omega)|
|\p_r\Omega^{\le1}f(t,r\omega)|d\omega dr\\
&\ls\int|\Omega^{\le1}f(t,y)||\w{|y|}\nabla\Omega^{\le1}f(t,y)|dy.
\end{split}
\end{equation*}
Thus, we derive the desired inequality \eqref{pw:prepare6}.
\end{proof}

\begin{lemma}\label{lem:pw}
For any multi-indices $a,b$ with $|a|\le N-3$ and $|b|\le N-2$, it holds that
\begin{equation}\label{pw:curl}
\w{|x|}^\frac32|\Gamma^aw(t,x)|\ls W_{|a|+2}(t),
\quad\w{|x|}^\frac32|\Gamma^a\curl u(t,x)|\ls\cW_{|a|+2}(t),
\end{equation}
\begin{equation}\label{pw:wave1}
\begin{split}
\w{|x|}^\frac12\w{|x|-t}^\frac12(|\tilde\Gamma^bP_1u(t,x)|+|\Gamma^b\sigma(t,x)|)
&\ls E_{|b|+2}(t)+\cX_{|b|+2}(t),\\
\w{|x|}^\frac12\w{|x|-t}(|\nabla\tilde\Gamma^aP_1u(t,x)|+|\nabla\Gamma^a\sigma(t,x)|)
&\ls E_{|a|+3}(t)+\cX_{|a|+3}(t).
\end{split}
\end{equation}
Furthermore, for $|x|\le3\w{t}/4$,
\begin{equation}\label{pw:wave2}
|\tilde\Gamma^bP_1u(t,x)|+|\Gamma^b\sigma(t,x)|
\ls\w{t}^{-1}\ln^\frac12(2+t)\big\{E_{|b|+2}(t)+\cX_{|b|+2}(t)\big\},
\end{equation}
\end{lemma}
\begin{proof}
The proof of \eqref{pw:curl} follows from \eqref{pw:prepare1} with $p=2$ directly.
The inequalities in \eqref{pw:wave1} can be concluded from \eqref{Helmholtz:id}, \eqref{div:curl:ineq},
\eqref{pw:prepare2} and \eqref{pw:prepare3}.

To achieve \eqref{pw:wave2}, we introduce the cutoff function $\chi(s)\in C^\infty$ such that
\begin{equation}\label{cutoff}
0\le\chi(s)\le1,\qquad \chi(s)=\left\{
\begin{aligned}
&1,\qquad\quad s\le3/4,\\
&0,\qquad\quad s\ge4/5.
\end{aligned}
\right.
\end{equation}
Choosing $f(t,x)=\chi\big(\frac{|x|}{\w{t}}\big)\tilde\Gamma^bP_1u(t,x)$ and $\chi\big(\frac{|x|}{\w{t}}\big)\Gamma^b\sigma(t,x)$ in \eqref{pw:prepare4}, we then get \eqref{pw:wave2} by \eqref{Helmholtz:id} and \eqref{div:curl:ineq}.
\end{proof}

\begin{lemma}\label{lem:pw:P2u}
For any multi-indices $a,b$ with $|a|\le N_1-1$ and $|b|\le N-2$, it holds that
\begin{align}
\w{|x|}^\frac83|P_2\tilde\Gamma^au(t,x)|&\ls\bW_{|a|+1}(t),\label{pw:P2u}\\
\w{|x|}^8|\Gamma^a\curl u(t,x)|&\ls\bW_{|a|+1}(t),\label{pw:curl:low}\\
\w{|x|}^7|\nabla\Gamma^aP_2u(t,x)|&\ls\bW_{|a|+1}(t),\label{pw:curl:low'}\\
\w{|x|}|P_2\tilde\Gamma^bu(t,x)|&\ls E_{|b|+2}(t)+\cW_{|b|+1}(t).\label{pw:P2u'}
\end{align}
\end{lemma}
\begin{proof}
Applying \eqref{pw:prepare5} to $P_2U$ yields that
\begin{equation*}
\begin{split}
\w{|x|}^\frac83|P_2U(t,x)|
&\ls\|\w{|y|}^7\nabla\Omega^{\le1}P_2U(t,y)\|_{L_y^5}
+\sum_{p=\frac{10}{9},\frac{10}{7}}\|\nabla\Omega^{\le1}P_2U(t,y)\|_{L_y^p}\\
&\ls\|\w{|y|}^7\Omega^{\le1}\curl U(t,y)\|_{L_y^5}
+\sum_{p=\frac{10}{9},\frac{10}{7}}\|\Omega^{\le1}\curl U(t,y)\|_{L_y^p},
\end{split}
\end{equation*}
where we have used \eqref{div:curl:ineq}. Subsequently, by choosing $U=\tilde\Gamma^au$,
\eqref{pw:P2u} is then derived.

The inequality \eqref{pw:curl:low} is a direct result of the Sobolev embedding $W^{1,5}(\R^2)\hookrightarrow L^\infty(\R^2)$.

Analogously, we can conclude from \eqref{div:curl:ineq} that
\begin{equation*}
\begin{split}
\w{|x|}^7|\nabla\Gamma^aP_2u(t,x)|
&\ls\|\w{|y|}^7\nabla\nabla^{\le1}\Gamma^aP_2u(t,y)\|_{L_y^5}\\
&\ls\|\w{|y|}^7\nabla^{\le1}\Gamma^a\curl u(t,y)\|_{L_y^5}\ls\bW_{|a|+1}(t),
\end{split}
\end{equation*}
which yields \eqref{pw:curl:low'}.

At last, we turn to the proof of \eqref{pw:P2u'}.
Let $f(t,x)=P_2\tilde\Gamma^bu(t,x)$ in \eqref{pw:prepare6} and then it concludes from \eqref{div:curl:ineq} that
\begin{equation*}
\begin{split}
\w{|x|}|P_2\tilde\Gamma^bu(t,x)|
&\ls E_{|b|+2}(t)+\|\w{x}\nabla P_2\Omega^{\le1}\tilde\Gamma^bu(t,x)\|_{L^2}\\
&\ls E_{|b|+2}(t)+\|\w{x}\curl\Omega^{\le1}\tilde\Gamma^bu(t,x)\|_{L^2},
\end{split}
\end{equation*}
which implies \eqref{pw:P2u'}.
\end{proof}
Combining Lemma \ref{lem:pw} and \ref{lem:pw:P2u}, we obtain the following pointwise estimates.
\begin{corollary}\label{coro:pw}
For any multi-indices $a,b$ with $|a|\le N_1-1$ and $|b|\le N-2$, it holds that
\begin{equation}\label{pw:wave3}
\begin{split}
|\Gamma^a\sigma(t,x)|+|\tilde\Gamma^au(t,x)|
&\ls\w{|x|}^{-\frac12}\w{|x|-t}^{-\frac12}\big\{E_{|a|+2}(t)+\cX_{|a|+2}(t)\big\}\\
&\quad+\w{|x|}^{-\frac83}\bW_{|a|+1}(t),\\
|\nabla\Gamma^a\sigma(t,x)|+|\nabla\tilde\Gamma^au(t,x)|
&\ls\w{|x|}^{-\frac12}\w{|x|-t}^{-1}\big\{E_{|a|+3}(t)+\cX_{|a|+3}(t)\big\}\\
&\quad+\w{|x|}^{-7}\bW_{|a|+1}(t),
\end{split}
\end{equation}
and
\begin{equation}\label{pw:wave3'}
\begin{split}
|\Gamma^b\sigma(t,x)|+|\tilde\Gamma^bu(t,x)|
&\ls\w{|x|}^{-\frac12}\w{|x|-t}^{-\frac12}\big\{E_{|b|+2}(t)+\cX_{|b|+2}(t)\big\}\\
&\quad+\w{|x|}^{-1}\big\{E_{|b|+2}(t)+\cW_{|b|+1}(t)\big\}.
\end{split}
\end{equation}
Furthermore, for $|x|\le3\w{t}/4$,
\begin{equation}\label{pw:wave4}
|\Gamma^a\sigma(t,x)|+|\tilde\Gamma^au(t,x)|\ls\w{|x|}^{-\frac83}\bW_{|a|+1}(t)
+\w{t}^{-1}\ln^\frac12(2+t)\big\{E_{|a|+2}(t)+\cX_{|a|+2}(t)\big\}.
\end{equation}
\end{corollary}

\section{Estimates of the good unknown and auxiliary energies}\label{sect3}

\subsection{Estimates of the good unknown $g$}\label{sect3:1}
In this subsection, several estimates of the good unknown $g$ will be established.
\begin{lemma}\label{lem:good:L2norm}
For $m\le N-1$, it holds that
\begin{equation}\label{good:L2norm}
\cG_m(t)\ls E_{m+1}(t)+\cX_{m+1}(t)+\cW_m(t)
+\sum_{|b|+|c|\le m}\|\w{|x|}Q_1^{bc}\|_{L^2(|x|\ge\w{t}/8)},
\end{equation}
where
\begin{equation}\label{good:L2norm:def}
\cG_m(t):=\sum_{|a|\le m}\|\w{|x|}\nabla\tilde\Gamma^ag(t,x)
\|^2_{L^2(|x|\ge\w{t}/8)}.
\end{equation}
\end{lemma}
\begin{proof}
Note that $r\p_r\tilde\Gamma^ag_i=x_j\p_j\tilde\Gamma^a(u-\sigma\omega)_i$, where the Einstein summation convention is used.
It follows from direct computation that there exist the bounded smooth functions $f_i^{a,b}(x)$ and
$f_{ij}^{a,b}(x)$ in $|x|\ge1/8$ such that
\begin{equation}\label{Gamma:sigma:omega}
\begin{split}
\tilde\Gamma^a(\sigma\omega)_i&=\omega_i\Gamma^a\sigma
+\frac{1}{|x|}\sum_{b+c\le a}f_i^{a,b}(x)\Gamma^c\sigma\\
\p_j\tilde\Gamma^a(\sigma\omega)_i&=\omega_i\p_j\Gamma^a\sigma
+\frac{1}{|x|}\sum_{b+c\le a}\Big[f_i^{a,b}(x)\p_j\Gamma^c\sigma
+f_{ij}^{a,b}(x)\Gamma^c\sigma\Big].
\end{split}
\end{equation}
Then we have
\begin{equation}\label{good:L2norm1}
\begin{split}
&\qquad r\p_r\tilde\Gamma^ag_i+\omega_j\sum_{b+c\le a}
[f_i^{a,b}(x)\p_j\Gamma^c\sigma+f_{ij}^{a,b}(x)\Gamma^c\sigma]\\
&=x_j\p_j\tilde\Gamma^au_i-\omega_ix_j\p_j\Gamma^a\sigma\\
&=x_j(\p_j\tilde\Gamma^au_i-\p_i\tilde\Gamma^au_j)+(x_j\p_i-x_i\p_j)\tilde\Gamma^au_j
+x_i\dive\tilde\Gamma^au-\omega_i\cS\Gamma^a\sigma+\omega_it\p_t\Gamma^a\sigma\\
&=x_j\eps_{ji}\Gamma^a\curl u+\eps_{ji}\Omega(\tilde\Gamma^au_j)
+x_i\cQ^a_1+\omega_i(t-|x|)\p_t\Gamma^a\sigma-\omega_i\cS\Gamma^a\sigma,
\end{split}
\end{equation}
where the volume form $\eps_{ji}$ is the sign of the arrangement $\{ji\}$ and we have
used the facts of $\p_jU_i-\p_iU_j=\eps_{ji}\curl U$ and $x_j\p_i-x_i\p_j=\eps_{ji}\Omega$.
Taking the $L^2(|x|\ge\w{t}/8)$ norm on the both sides of \eqref{good:L2norm1} yields
\begin{equation}\label{good:L2norm2}
\begin{split}
\sum_{|a|\le m}\|\w{|x|}\p_r\tilde\Gamma^ag(t,x)\|_{L^2(|x|\ge\w{t}/8)}
&\ls E_{m+1}(t)+\cX_{m+1}(t)+\cW_m(t)\\
&\quad+\sum_{|b|+|c|\le m}\|\w{|x|}Q_1^{bc}\|_{L^2(|x|\ge\w{t}/8)}.
\end{split}
\end{equation}
In addition, it follows from \eqref{radial:angular} that
\begin{equation}\label{good:L2norm3}
\cG_m(t)
\ls\sum_{|a|\le m}\|\w{|x|}\p_r\tilde\Gamma^ag\|_{L^2(|x|\ge\w{t}/8)}+E_{m+1}(t).
\end{equation}
Collecting \eqref{good:L2norm2} and \eqref{good:L2norm3} together leads to \eqref{good:L2norm}.
\end{proof}

\begin{lemma}\label{lem:good:pw}
For $|a|\le N-1$, $|b|\le N-2$ and $|x|\ge\w{t}/8$, it holds that
\begin{equation}\label{good:pw}
\begin{split}
\w{|x|}|\tilde\Gamma^ag(t,x)|&\ls\cG_{|a|+1}(t)+E_{|a|+1}(t),\\
\w{|x|}^\frac32|\nabla\tilde\Gamma^bg(t,x)|&\ls\cG_{|b|+2}(t).
\end{split}
\end{equation}
\end{lemma}
\begin{proof}
Applying the Sobolev embedding on the unit circle and the Newton-Leibnitz formula in the radial direction derive that
\begin{equation*}
\begin{split}
\w{|x|}^{1+\ell}|U(t,x)|^2&\ls\w{|x|}^{1+\ell}\int_{\SS^1}
|\tilde\Omega^{\le1}U(t,|x|\omega)|^2d\omega\\
&\ls\w{|x|}^\ell\int_{|x|}^\infty\int_{\SS^1}
|\tilde\Omega^{\le1}U(t,r\omega)\p_r\tilde\Omega^{\le1}U(t,r\omega)|rd\omega dr.
\end{split}
\end{equation*}
Choosing $U(t,x)=\tilde\Gamma^ag(t,x), \nabla\tilde\Gamma^bg(t,x)$ in the above equality with $\ell=1,2$, respectively,
we then get that for $|x|\ge\w{t}/8$,
\begin{equation*}
\begin{split}
\w{|x|}^2|\tilde\Gamma^ag(t,x)|^2&\ls
\|\tilde\Omega^{\le1}\tilde\Gamma^ag(t,y)\|^2_{L^2(|y|\ge\w{t}/8)}
+\|\w{|y|}\nabla\tilde\Omega^{\le1}\tilde\Gamma^ag(t,y)\|^2_{L^2(|y|\ge\w{t}/8)},\\
\w{|x|}^3|\nabla\tilde\Gamma^bg(t,x)|^2&\ls
\|\w{|y|}\nabla\nabla^{\le1}\tilde\Omega^{\le1}\tilde\Gamma^bg(t,y)\|^2_{L^2(|y|\ge\w{t}/8)}.
\end{split}
\end{equation*}
This completes the proof of Lemma~\ref{lem:good:pw}.
\end{proof}

Based on Lemma \ref{lem:good:L2norm} and \ref{lem:good:pw}, we have the following estimates.
\begin{lemma}\label{lem:Qbc:L2}
Under bootstrap assumptions \eqref{bootstrap}, it holds that for $m\le N-1$,
\begin{equation}\label{Qbc:L2}
\sum_{|b|+|c|\le m}\|\w{|x|}(|Q_1^{bc}|+|Q_2^{bc}|)\|_{L^2(|x|\ge\w{t}/8)}
\ls M\ve\cG_m(t)+E_{m+1}(t)[1+\cG_{N_1-5}(t)].
\end{equation}
\end{lemma}
\begin{proof}
According to \eqref{goodunknown:def} and equalities \eqref{Gamma:sigma:omega}, one easily gets
\begin{equation}\label{null:structure}
\begin{split}
\tilde\Gamma^bu_i&=\tilde\Gamma^bg_i+\omega_i\Gamma^b\sigma
+\frac{1}{|x|}\sum_{b_1+b_2\le b}f_i^{b,b_1}(x)\Gamma^{b_2}\sigma,\\
\p_j\tilde\Gamma^cu_i&=\p_j\tilde\Gamma^cg_i+\omega_i\p_j\Gamma^c\sigma
+\frac{1}{|x|}\sum_{c_1+c_2\le c}\Big[f_i^{c,c_1}(x)\p_j\Gamma^{c_2}\sigma
+f_{ij}^{c,c_1}(x)\Gamma^{c_2}\sigma\Big].
\end{split}
\end{equation}
Substituting \eqref{radial:angular} and \eqref{null:structure} into \eqref{Qbc:def} yields
\begin{equation}\label{null:Q1}
\begin{split}
Q_1^{bc}&=\Gamma^b\sigma\Big\{\p_i\tilde\Gamma^cg_i
+\frac{1}{|x|}\sum_{c_1+c_2\le c}\Big[f_i^{c,c_1}(x)\p_i\Gamma^{c_2}\sigma
+f_{ii}^{c,c_1}(x)\Gamma^{c_2}\sigma\Big]\Big\}\\
&\quad-\p_i\Gamma^c\sigma\Big\{\tilde\Gamma^bg_i
+\frac{1}{|x|}\sum_{b_1+b_2\le b}f_i^{b,b_1}(x)\Gamma^{b_2}\sigma\Big\},
\end{split}
\end{equation}
and
\begin{equation}\label{null:Q2}
\begin{split}
Q_{2i}^{bc}&=\frac{1}{|x|}\Gamma^b\sigma\Omega\Gamma^c\sigma
-\omega_i\p_j\Gamma^c\sigma\Big\{\tilde\Gamma^bg_j
+\frac{1}{|x|}\sum_{b_1+b_2\le b}f_j^{b_1}(x)\Gamma^{b_2}\sigma\Big\}\\
&\quad-\tilde\Gamma^bu_j\Big\{\p_j\tilde\Gamma^cg_i
+\frac{1}{|x|}\sum_{c_1+c_2\le c}\Big[f_i^{c_1}(x)\p_j\Gamma^{c_2}\sigma
+f_{ij}^{c_1}(x)\Gamma^{c_2}\sigma\Big]\Big\}.
\end{split}
\end{equation}
By applying the pointwise estimates \eqref{pw:wave3} to the terms that containing the factor
$\frac{1}{|x|}$ in \eqref{null:Q1} and \eqref{null:Q2} directly, we obtain that
\begin{equation}\label{Qbc:L2:1}
\begin{split}
\|\w{|x|}(|Q_1^{bc}|+|Q_2^{bc}|)\|_{L^2(|x|\ge\w{t}/8)}
&\ls E_{m+1}(t)+\|\w{|x|}|\nabla\tilde\Gamma^cg|
(|\Gamma^b\sigma|+|\tilde\Gamma^bu|)\|_{L^2(|x|\ge\w{t}/8)}\\
&\quad+\|\w{|x|}\tilde\Gamma^bg\nabla\Gamma^c\sigma\|_{L^2(|x|\ge\w{t}/8)}.
\end{split}
\end{equation}
At last, by the virtue of the estimates in Lemma \ref{lem:good:L2norm} and \ref{lem:good:pw},
we will deal with the remaining terms
in \eqref{null:Q1} and \eqref{null:Q2}.

Due to $|b|+|c|\le m\le N-1\le2N_1-12$, then $|b|\le N_1-6$ or $|c|\le N_1-7$ holds.
For $|b|\le N_1-6$, it follows from \eqref{bootstrap} and \eqref{pw:wave3} that
\begin{equation*}
\Big\||\Gamma^b\sigma|+|\tilde\Gamma^bu|\Big\|_{L^\infty(|x|\ge\w{t}/8)}
\ls\w{t}^{-\frac12}\big\{E_{|b|+2}(t)+\cX_{|b|+2}(t)+\bW_{|b|+1}(t)\big\}\ls M\ve.
\end{equation*}
This together with \eqref{good:pw} implies that
\begin{equation}\label{Qbc:L2:2}
\begin{split}
&\|\w{|x|}\tilde\Gamma^bg\nabla\Gamma^c\sigma\|_{L^2(|x|\ge\w{t}/8)}
+\|\w{|x|}|\nabla\tilde\Gamma^cg|(|\Gamma^b\sigma|+|\tilde\Gamma^bu|)
\|_{L^2(|x|\ge\w{t}/8)}\\
&\ls E_{|c|+1}(t)[\cG_{|b|+1}(t)+E_{|b|+1}(t)]+M\ve\cG_{|c|}(t)\\
&\ls M\ve\cG_m(t)+E_{m+1}(t)[1+\cG_{N_1-5}(t)].
\end{split}
\end{equation}
For $|c|\le N_1-7$, by \eqref{pw:wave1}, we have
\begin{equation*}
\begin{split}
\|\nabla\Gamma^c\sigma\|_{L^\infty(|x|\ge\w{t}/8)}
&\ls\w{|x|}^{-\frac12}\w{|x|-t}^{-1}\big\{E_{|c|+3}(t)+\cX_{|c|+3}(t)\big\}\\
&\ls M\ve\w{|x|}^{-\frac12}\w{|x|-t}^{-1}.
\end{split}
\end{equation*}
Therefore,
\begin{equation}\label{Qbc:L2:3}
\begin{split}
&\|\w{|x|}\tilde\Gamma^bg\nabla\Gamma^c\sigma\|_{L^2(|x|\ge\w{t}/8)}
+\|\w{|x|}|\nabla\tilde\Gamma^cg|(|\Gamma^b\sigma|+|\tilde\Gamma^bu|)
\|_{L^2(|x|\ge\w{t}/8)}\\
&\ls E_{|b|}(t)[\cG_{|c|+2}(t)+E_{|c|+2}(t)]+E_{|b|}(t)
+M\ve\|\w{|x|}^\frac12\w{|x|-t}^{-1}\tilde\Gamma^bg\|_{L^2(|x|\ge4\w{t}/5)}.
\end{split}
\end{equation}
For the last term in \eqref{Qbc:L2:3}, performing the integration by parts for the radial direction yields
\begin{equation}\label{Qbc:L2:4}
\begin{split}
&\|\w{|x|}^\frac12\w{|x|-t}^{-1}\tilde\Gamma^bg\|^2_{L^2(|x|\ge4\w{t}/5)}\\
&\ls\int_0^\infty\int_{\SS^1}|\tilde\Gamma^bg(t,r\omega)|^2\w{r}r
\Big[1-\chi\big(\frac{r}{\w{t}}\big)\Big]d\omega d\arctan(r-t)\\
&\ls\cG^2_{|b|}(t)+E^2_{|b|}(t),
\end{split}
\end{equation}
where the cutoff function $\chi$ is defined by \eqref{cutoff}.

Substituting \eqref{Qbc:L2:2}--\eqref{Qbc:L2:4} into \eqref{Qbc:L2:1} derives \eqref{Qbc:L2}.
This completes the proof of Lemma~\ref{lem:Qbc:L2}.
\end{proof}

Combining Lemma \ref{lem:good:L2norm}--\ref{lem:Qbc:L2}, we obtain the following result.
\begin{corollary}\label{coro:good}
Under bootstrap assumptions \eqref{bootstrap}, for $|a|\le N-1$, $|b|\le N-2$ and $|x|\ge\w{t}/8$, it holds that
\begin{equation}\label{good:pw1}
\begin{split}
\w{|x|}|\tilde\Gamma^ag(t,x)|&\ls E_{|a|+1}(t)+\cX_{|a|+1}(t)+\cW_{|a|}(t),\\
\w{|x|}^\frac32|\nabla\tilde\Gamma^bg(t,x)|&\ls E_{|b|+2}(t)+\cX_{|b|+2}(t)+\cW_{|b|+1}(t).
\end{split}
\end{equation}
Moreover, for any integer $m$ with $0\le m\le N-1$, it holds that
\begin{equation}\label{Qbc:L2'}
\sum_{|b|+|c|\le m}\|\w{|x|}(|Q_1^{bc}|+|Q_2^{bc}|)\|_{L^2(|x|\ge\w{t}/8)}
\ls E_{m+1}(t)+M\ve[\cX_{m+1}(t)+\cW_m(t)].
\end{equation}
\begin{proof}
It concludes from \eqref{good:L2norm} and \eqref{Qbc:L2} with $m=N_1-5$, \eqref{bootstrap} and the smallness of $M\ve$ that
\begin{equation*}
\cG_{N_1-5}(t)\ls E_{N_1-4}(t)[1+\cG_{N_1-5}(t)]+\cX_{N_1-4}(t)+\cW_{N_1-5}(t),
\end{equation*}
which implies $\cG_{N_1-5}(t)\ls M\ve$.
Together with \eqref{good:L2norm} and \eqref{Qbc:L2} again, this yields
\begin{equation}\label{goodL2norm2}
\cG_m(t)\ls E_{m+1}(t)+\cX_{m+1}(t)+\cW_m(t).
\end{equation}
Substituting \eqref{goodL2norm2} into \eqref{good:pw} and \eqref{Qbc:L2} completes the proof of Corollary~\ref{coro:good}.
\end{proof}

\end{corollary}
\subsection{Estimates of the auxiliary energy $\cX_m(t)$}
\begin{lemma}[Weighted $\dot H_x^1$]\label{lem:H1norm}
Under bootstrap assumptions \eqref{bootstrap}, for any integer $m$ with $1\le m\le N$, it holds that
\begin{equation}\label{H1norm}
\cX_m(t)\ls E_m(t)+\cW_{m-1}(t).
\end{equation}
\end{lemma}
\begin{proof}
For $|a|\le m-1$, it follows from direct computations and equations \eqref{high:eqn} that
\begin{equation}\label{weighted:identity1}
\begin{split}
(|x|^2-t^2)\p_t\tilde\Gamma^au_i&=|x|^2(\cQ^a_{2i}-\p_i\Gamma^a\sigma)
-t\cS\tilde\Gamma^au_i+tx_j\p_j\tilde\Gamma^au_i\\
&=|x|^2\cQ^a_{2i}-x_j(x_j\p_i-x_i\p_j)\Gamma^a\sigma-x_i\cS\Gamma^a\sigma
+tx_i\p_t\Gamma^a\sigma-t\cS\tilde\Gamma^au_i\\
&\quad+tx_j(\p_j\tilde\Gamma^au_i-\p_i\tilde\Gamma^au_j)
+t(x_j\p_i-x_i\p_j)\tilde\Gamma^au_j+tx_i\dive\tilde\Gamma^au\\
&=|x|^2\cQ^a_{2i}-x_j\eps_{ji}\Omega\Gamma^a\sigma-x_i\cS\Gamma^a\sigma
+tx_i\cQ_1^a-t\cS\tilde\Gamma^au_i\\
&\quad+tx_j\eps_{ji}\Gamma^a\curl u+t\eps_{ji}\Omega(\tilde\Gamma^au_j).
\end{split}
\end{equation}
Here we point out that the main difference between \eqref{weighted:identity1} and the analogous equality of $\p_tP_1\tilde\Gamma^au_i$ in \cite{Sideris97} lies in the presence of the vorticity $\Gamma^a\curl u$ in \eqref{weighted:identity1}.

On the other hand, we can obtain
\begin{equation}\label{weighted:identity2}
\begin{split}
(|x|^2-t^2)\p_t\Gamma^a\sigma&=|x|^2\cQ^a_1-x_j\eps_{ji}\Omega(\tilde\Gamma^au_i)
-x_i\cS\tilde\Gamma^au_i-t\cS\Gamma^a\sigma+tx_i\cQ^a_{2i},\\
(|x|^2-t^2)\p_i\Gamma^a\sigma&=x_j\eps_{ji}\Omega\Gamma^a\sigma+x_i\cS\Gamma^a\sigma
-tx_i\cQ^a_1-t^2\cQ^a_{2i}+t\cS\tilde\Gamma^au_i\\
&\quad+tx_j\eps_{ij}\Gamma^a\curl u+t\eps_{ij}\Omega(\tilde\Gamma^au_j),\\
(|x|^2-t^2)\dive\tilde\Gamma^au&=x_j\eps_{ji}\Omega(\tilde\Gamma^au_i)
+x_i\cS\tilde\Gamma^au_i-tx_i\cQ^a_{2i}-t^2\cQ^a_1+t\cS\Gamma^a\sigma.
\end{split}
\end{equation}
In view of $\w{|x|-t}\ls1+||x|-t|$, by dividing $|x|+t$ and then taking $L_x^2$ norm on the both sides of \eqref{weighted:identity1} and \eqref{weighted:identity2}, we arrive at
\begin{equation}\label{H1norm1}
\cX_m(t)\ls E_m(t)+\cW_{m-1}(t)+\sum_{|b|+|c|\le m-1}
\|\w{|x|+t}(|Q_1^{bc}|+|Q_2^{bc}|)\|_{L_x^2},
\end{equation}
where $Q_1^{bc},Q_2^{bc}$ are defined in \eqref{Qbc:def}.

We next investigate the $L_x^2$ norms of $Q_1^{bc}$ and $Q_2^{bc}$, which are divided into the
two parts of $|x|\ge\w{t}/8$ and $|x|\le\w{t}/8$.

It is easy to deduce from \eqref{Qbc:L2'} that
\begin{equation}\label{H1norm2}
\sum_{|b|+|c|\le m-1}\|\w{|x|+t}(|Q_1^{bc}|+|Q_2^{bc}|)\|_{L^2(|x|\ge\w{t}/8)}
\ls E_m(t)+\cW_{m-1}(t)+M\ve\cX_m(t).
\end{equation}
We now deal with $\|\w{|x|-t}(|Q_1^{bc}|+|Q_2^{bc}|)\|_{L^2(|x|\le\w{t}/8)}$. In fact,
only $\tilde\Gamma^bu\cdot\nabla\tilde\Gamma^cu$ requires to be treated  since the treatments on the other terms $\Gamma^b\sigma\dive\tilde\Gamma^cu$, $\tilde\Gamma^bu\cdot\nabla\Gamma^c\sigma$,
$\Gamma^b\sigma\nabla\Gamma^c\sigma$ in $Q_1^{bc}$ and $Q_2^{bc}$ are analogous.

Similarly to Lemma~\ref{lem:Qbc:L2}, it always holds that $|b|\le N_1-5$ or $|c|\le N_1-7$.
For the case of $|c|\le N_1-7$, applying \eqref{pw:wave3} to $\nabla\tilde\Gamma^cu$ leads to
\begin{equation}\label{H1norm3}
\|\w{|x|-t}\tilde\Gamma^bu\cdot\nabla\tilde\Gamma^cu\|_{L^2(|x|\le\w{t}/8)}
\ls E_{|b|}(t)[\w{t}\bW_{|c|+1}(t)+E_{|c|+3}(t)+\cX_{|c|+3}(t)]\ls E_m(t),
\end{equation}
where we have used assumptions \eqref{bootstrap}.

For the case of $|b|\le\min\{N_1-5,m-2\}$, by utilizing \eqref{pw:wave4} to $\tilde\Gamma^bu$, we can see that
\begin{equation}\label{H1norm4}
\begin{split}
&\quad\|\w{|x|-t}\tilde\Gamma^bu\cdot\nabla\tilde\Gamma^cu\|_{L^2(|x|\le\w{t}/8)}\\
&\ls\|\tilde\Gamma^bu\|_{L^\infty(|x|\le\w{t}/8)}
\Big\|\w{|x|-t}\chi\big(\frac{|x|}{\w{t}}\big)\nabla\Gamma^cu\Big\|_{L^2}\\
&\ls\Big\|\w{|x|-t}\chi\big(\frac{|x|}{\w{t}}\big)\nabla\Gamma^cu\Big\|_{L^2}
\Big\{\bW_{|b|+1}(t)+\w{t}^{-1}\ln^\frac12(2+t)[E_{|b|+2}(t)+\cX_{|b|+2}(t)]\Big\}\\
&\ls\Big\|\w{|x|-t}\chi\big(\frac{|x|}{\w{t}}\big)\nabla\Gamma^cu\Big\|_{L^2}
\Big\{\bW_{N_1-4}(t)+\w{t}^{-1}\ln^\frac12(2+t)[E_m(t)+\cX_m(t)]\Big\}.
\end{split}
\end{equation}
In addition, it follows from \eqref{div:curl:ineq} that
\begin{equation}\label{H1norm5}
\begin{split}
&\quad\Big\|\w{|x|-t}\chi\big(\frac{|x|}{\w{t}}\big)\nabla\Gamma^cu\Big\|_{L^2}\\
&\ls E_{|c|}(t)+\Big\|\w{|x|-t}\nabla\Big(\chi\big(\frac{|x|}{\w{t}}\big)
\Gamma^cu\Big)\Big\|_{L^2}\\
&\ls E_{|c|}(t)+\|\w{|x|-t}\dive\Gamma^cu\|_{L^2}
+\w{t}\|\w{|x|}\curl\Gamma^cu\|_{L^2}\\
&\ls E_{|c|}(t)+\cX_{|c|+1}(t)+\w{t}\cW_{|c|}(t)\\
&\ls E_m(t)+\cX_m(t)+\w{t}\cW_{m-1}(t).
\end{split}
\end{equation}
Substituting \eqref{H1norm5} into \eqref{H1norm4} derives
\begin{equation}\label{H1norm6}
\begin{split}
&\quad\|\w{|x|-t}\tilde\Gamma^bu\cdot\nabla\tilde\Gamma^cu\|_{L^2(|x|\le\w{t}/8)}\\
&\ls\big\{E_m(t)+\cX_m(t)+\w{t}\cW_{m-1}(t)\big\}
\big\{M\delta+\w{t}^{-1}\ln^\frac12(2+t)[E_m(t)+\cX_m(t)]\big\}\\
&\ls E_m(t)+M\ve\cX_m(t)+\cW_{m-1}(t),
\end{split}
\end{equation}
where we have also used assumptions \eqref{bootstrap}.

Collecting \eqref{H1norm3} and \eqref{H1norm6} yields that
\begin{equation}\label{H1norm7}
\sum_{|b|+|c|\le m-1}
\|\w{|x|+t}(|Q_1^{bc}|+|Q_2^{bc}|)\|_{L^2(|x|\le\w{t}/8)}
\ls E_m(t)+M\ve\cX_m(t)+\cW_{m-1}(t).
\end{equation}
By combining \eqref{H1norm2} and \eqref{H1norm7}, we eventually achieve
\begin{equation}\label{H1norm8}
\sum_{|b|+|c|\le m-1}
\|\w{|x|+t}(|Q_1^{bc}|+|Q_2^{bc}|)\|_{L_x^2}\ls E_m(t)+M\ve\cX_m(t)+\cW_{m-1}(t).
\end{equation}
Plugging \eqref{H1norm8} into \eqref{H1norm1} with the smallness of $M\ve$, then \eqref{H1norm} is proved.
\end{proof}

\section{Improved pointwise estimates}\label{sect4}
Note that the decay rate of the irrotational part $P_1u$ of the velocity $u$
is merely $\ve\w{t}^{-1}$   away from the light cone  (see Lemma \ref{lem:pw}).
This is far to achieve the desired lifespan $T_{\dl}=O(\frac{1}{\delta})$, for examples, when
$\dl=e^{-\f{1}{\ve^2}}$ or $\dl=e^{-e^{\f{1}{\ve^2}}}$ are chosen, whose reason has been explained in
Section 1. It is required to improve the related pointwise estimates in Section \ref{sect3}.

\subsection{Improved pointwise estimates of the first order derivatives of $(\sigma, u)$}
In this subsection, by the virtue of the weighted identities \eqref{good:L2norm1}, \eqref{weighted:identity1} and \eqref{weighted:identity2}, the pointwise estimates of $\nabla\tilde\Gamma^au$, $\nabla\Gamma^a\sigma$ in \eqref{pw:wave3}, \eqref{pw:wave4} and $\nabla\tilde\Gamma^ag$ in \eqref{good:pw1} can be improved as follows.
\begin{lemma}\label{lem:pw:improve}
Under bootstrap assumptions \eqref{bootstrap}, if $|a|\le N_1-1$, then for $|x|\ge\w{t}/8$, it holds that
\begin{equation}\label{pw:wave5}
\begin{split}
|\p_t\tilde\Gamma^au(t,x)|+|\p_t\Gamma^a\sigma(t,x)|
+|\dive\tilde\Gamma^au(t,x)|+|\nabla\Gamma^a\sigma(t,x)|
&\ls M\ve\w{|x|}^{2M'\ve-\frac12}\w{|x|-t}^{-\frac32},\\
|\nabla\tilde\Gamma^au(t,x)|
&\ls M\ve\w{|x|}^{2M'\ve-\frac12}\w{|x|-t}^{-\frac32},
\end{split}
\end{equation}
and
\begin{equation}\label{good:pw2}
|\nabla\tilde\Gamma^ag(t,x)|\ls M\ve\w{|x|}^{2M'\ve-\frac32}\w{|x|-t}^{-\frac12}.
\end{equation}
On the other hand, for $|x|\le3\w{t}/4$, we have
\begin{equation}\label{pw:wave6}
\begin{split}
|\p_t\tilde\Gamma^au(t,x)|+|\p_t\Gamma^a\sigma(t,x)|
&\ls M\delta\w{t}^{M'\ve-1}+M\ve\w{t}^{2M'\ve-2}\ln(2+t),\\
|\dive\tilde\Gamma^au(t,x)|+|\nabla\Gamma^a\sigma(t,x)|
&\ls M\delta\w{t}^{M'\ve-1}+M\ve\w{t}^{2M'\ve-2}\ln(2+t),\\
\w{|x|}|\nabla\tilde\Gamma^au(t,x)|
&\ls M\delta\w{t}^{M'\ve}+M\ve\w{t}^{2M'\ve-1}\ln(2+t).
\end{split}
\end{equation}

\end{lemma}
\begin{proof}
For $|x|\ge\w{t}/8$, it follows from \eqref{pw:curl:low}, \eqref{pw:wave3}, \eqref{weighted:identity1} and \eqref{weighted:identity2} that
\begin{equation}\label{improve:pw1}
\begin{split}
&\quad\w{|x|-t}(|\p_t\tilde\Gamma^au(t,x)|+|\p_t\Gamma^a\sigma(t,x)|
+|\dive\tilde\Gamma^au(t,x)|+|\nabla\Gamma^a\sigma(t,x)|)\\
&\ls\sum_{|b|\le|a|+1}(|\Gamma^b\sigma|+|\tilde\Gamma^bu|)
+\w{|x|}|\Gamma^a\curl u|+\w{|x|}\sum_{b+c\le a}(|Q_1^{bc}|+|Q_2^{bc}|)\\
&\ls\bW_{|a|+1}(t)\w{|x|}^{-\frac83}
+\w{|x|}^{-\frac12}\w{|x|-t}^{-\frac12}\Big\{E_{|a|+3}(t)+\cX_{|a|+3}(t)\Big\}\\
&\quad+\bW_{|a|+1}(t)\w{|x|}^{-9}+\w{|x|}\sum_{b+c\le a}(|Q_1^{bc}|+|Q_2^{bc}|).
\end{split}
\end{equation}
Applying \eqref{pw:wave3} and \eqref{good:pw1} to \eqref{null:Q1} and \eqref{null:Q2} yields
\begin{equation}\label{improve:pw2}
\begin{split}
|Q_1^{bc}|+|Q_2^{bc}|&\ls\w{|x|}^{-1}
\Big(M\delta\w{|x|}^{M'\ve-\frac83}+M\ve\w{|x|}^{M'\ve-\frac12}\w{|x|-t}^{-\frac12}\Big)^2\\
&\quad+|\tilde\Gamma^bg\nabla\Gamma^c\sigma|
+|\nabla\tilde\Gamma^cg|(|\Gamma^b\sigma|+|\tilde\Gamma^bu|)\\
&\ls M\delta\w{|x|}^{2M'\ve-\frac{11}{3}}
+M^2\ve^2\w{|x|}^{2M'\ve-\frac32}\w{|x|-t}^{-1},
\end{split}
\end{equation}
where we have used the bootstrap assumptions \eqref{bootstrap}.
Substituting \eqref{improve:pw2} into \eqref{improve:pw1} infers
\begin{equation*}
\begin{split}
&\quad\w{|x|-t}(|\p_t\tilde\Gamma^au(t,x)|+|\p_t\Gamma^a\sigma(t,x)|
+|\dive\tilde\Gamma^au(t,x)|+|\nabla\Gamma^a\sigma(t,x)|)\\
&\ls M\delta\w{|x|}^{2M'\ve-\frac83}+M\ve\w{|x|}^{M'\ve-\frac12}\w{|x|-t}^{-\frac12}
+M^2\ve^2\w{|x|}^{2M'\ve-\frac12}\w{|x|-t}^{-1}\\
&\ls M\ve\w{|x|}^{2M'\ve-\frac12}\w{|x|-t}^{-\frac12}.
\end{split}
\end{equation*}
This leads to the first inequality in \eqref{pw:wave5}.

Next, we turn to the proof of the second inequality in \eqref{pw:wave5}.
By \eqref{radial:angular}, we have
\begin{equation}\label{improve:pw2'}
\w{|x|}|\nabla\tilde\Gamma^au(t,x)|
\ls|\Gamma^{\le1}\tilde\Gamma^au(t,x)|+|r\p_r\tilde\Gamma^au(t,x)|
\ls|\Gamma^{\le1}\tilde\Gamma^au(t,x)|+|t\p_t\tilde\Gamma^au(t,x)|.
\end{equation}
This, together with the estimate of $\p_t\tilde\Gamma^au(t,x)$ in \eqref{pw:wave5}, yields the second inequality in \eqref{pw:wave5}.

It is not hard to conclude from \eqref{good:L2norm1} that for $|x|\ge\w{t}/8$,
\begin{equation}\label{improve:pw3}
\begin{split}
&\quad\w{|x|}|\nabla\tilde\Gamma^ag(t,x)|\\
&\ls\w{|x|}|\Gamma^a\curl u|
+\sum_{|b|\le|a|+1}(|\tilde\Gamma^bg|+|\Gamma^b\sigma|+|\tilde\Gamma^bu|)
+\w{|x|}|\cQ^a_1|+\w{|x|-t}|\p_t\Gamma^a\sigma|\\
&\ls\w{|x|}^{-\frac83}\bW_{|a|+1}(t)
+\w{|x|}^{-\frac12}\w{|x|-t}^{-\frac12}\Big\{E_{|a|+3}(t)+\cX_{|a|+3}(t)\Big\}\\
&\quad+\w{|x|}|\cQ^a_1|+\w{|x|-t}|\p_t\Gamma^a\sigma|,
\end{split}
\end{equation}
where we have used \eqref{pw:curl:low} and \eqref{pw:wave3}.

Combining \eqref{improve:pw3} with \eqref{pw:wave5} and \eqref{improve:pw2} yields \eqref{good:pw2}.

Finally, we turn to the proof of \eqref{pw:wave6}.
For $|x|\le3\w{t}/4$, by using \eqref{pw:curl:low} and \eqref{pw:wave4} to \eqref{weighted:identity1} and \eqref{weighted:identity2} directly, we arrive at
\begin{equation*}
\begin{split}
&\quad\w{t}(|\p_t\tilde\Gamma^au(t,x)|+|\p_t\Gamma^a\sigma(t,x)|
+|\dive\tilde\Gamma^au(t,x)|+|\nabla\Gamma^a\sigma(t,x)|)\\
&\ls\sum_{|b|\le|a|+1}(|\Gamma^b\sigma|+|\tilde\Gamma^bu|)
+\w{|x|}|\Gamma^a\curl u|+\w{t}\sum_{b+c\le a}(|Q_1^{bc}|+|Q_2^{bc}|)\\
&\ls\bW_{|a|+1}(t)
+\w{t}^{-1}\ln^\frac12(2+t)\Big\{E_{|a|+3}(t)+\cX_{|a|+3}(t)\Big\}\\
&\quad+\w{t}^{-1}\ln(2+t)\Big(E_{|a|+3}(t)+\cX_{|a|+3}(t)\Big)^2\\
&\ls M\delta\w{t}^{M'\ve}+M\ve\w{t}^{2M'\ve-1}\ln(2+t).
\end{split}
\end{equation*}
Then we can achieve the first two inequalities in \eqref{pw:wave6}.
Combining \eqref{improve:pw2'} with the estimates $\p_t\tilde\Gamma^au(t,x)$ in \eqref{pw:wave6},
we get the third inequality in \eqref{pw:wave6}.
\end{proof}

\subsection{Weighted $L^\infty$-$L^\infty$ estimates for the linear wave equation}\label{sect:weight:pw}
In this subsection, we will establish some weighted $L^\infty$-$L^\infty$ estimates for
the solutions to the linear wave equations.
Consider the following Cauchy problem
\begin{equation}\label{linear:wave}
\Box\vp:=\p_t^2\vp-\Delta\vp=\cF,\qquad(\vp,\p_t\vp)|_{t=0}=(\vp_0,\vp_1).
\end{equation}
Then $\vp=\vp_{hom}+\vp_{inh}$, where
\begin{equation}\label{phi:hom:def}
\Box\vp_{hom}=0,\qquad(\vp_{hom},\p_t\vp_{hom})|_{t=0}=(\vp_0,\vp_1),
\end{equation}
and
\begin{equation*}
\Box\vp_{inh}=\cF,\qquad(\vp_{inh},\p_t\vp_{inh})|_{t=0}=(0,0).
\end{equation*}

\begin{lemma}\label{H-0}[Proposition 4.1 and 4.2 of \cite{HoshigaKubo04}]
Let $0<\nu<\frac12$ and $\mu>0$, then it holds that
\begin{align}
&\w{|x|+t}^\frac12\w{|x|-t}^\nu|\vp_{inh}(t,x)|\ls
\tilde\cM_{\mu+\nu}(\cF)(t),\label{sharp:pw1}\\
&\w{|x|}^\frac12\w{|x|-t}^{1+\nu}|\nabla\vp_{inh}(t,x)|
\ls\sum_{|a|+j\le1}\tilde\cM_{\mu+\nu}(\nabla^a\Omega^j\cF)(t),\label{sharp:pw2}
\end{align}
where
\begin{equation}\label{tildecM:def}
\begin{split}
\tilde\cM_{\mu+\nu}(\cF)(t)=&\sup_{(s,y)\in\Lambda_0(t)}
\{\w{|y|}^\frac32\w{|y|+s}^{1+\mu+\nu}|\cF(s,y)|\}\\
&+\sup_{(s,y)\in\Lambda_1(t)}\{\w{s}^{\frac32+\mu+\nu}\w{|y|-s}|\cF(s,y)|\},
\end{split}
\end{equation}
and
\begin{equation}\label{Lambda:def}
\begin{split}
&\Lambda_1(t)=\{(s,y)\in[0,t]\times\R^2:||y|-s|\le s/3,|y|\ge1\},\\
&\Lambda_0(t)=[0,t]\times\R^2\setminus\Lambda_1(t)=\{(s,y)\in[0,t]
\times\R^2:||y|-s|\ge s/3,{\rm~or~}|y|\le1\}.
\end{split}
\end{equation}

\end{lemma}
\begin{remark}
The notation $\tilde\cM_{\mu+\nu}(\cF)(t)$ on the right hand side of \eqref{sharp:pw1}
and \eqref{sharp:pw2} is slightly different from that in\cite{HoshigaKubo04}, in which is $M_{\nu}(\cF)(t)$ with
$$M_{\nu}(\cF)(t)=\ds\sup_{(s,y)\in\Lambda_0(t)}
\{\w{|y|}^\frac32\w{|y|+s}^{1+\mu+\nu}|\cF(s,y)|\} +\ds\sup_{(s,y)\in\Lambda_1(t)}\{\w{s}^{\frac32+\mu+\nu}\w{|y|-s}|\cF(s,y)|\}.$$
\end{remark}
Unfortunately, Lemma \ref{H-0} can not be applied directly for our problem. We next give a modified version as follows.
\begin{lemma}\label{lem:weight:inh}
Let $0<\mu_1,\nu<\frac12$ and $\mu>0$, then for $|x|\le2\w{t}$, it holds that
\begin{align}
\w{|x|+t}^{\frac12-\mu_1}\w{|x|-t}^\nu|\vp_{inh}(t,x)|
&\ls\cM_{\mu+\nu-\mu_1}(\cF)(t),\label{sharp:pw1'}\\
\w{|x|}^\frac12\w{|x|-t}^{1+\nu}|\nabla\vp_{inh}(t,x)|
&\ls\sum_{|a|+j\le1}\cM_{\mu+\nu}(\nabla^a\Omega^j\cF)(t),\label{sharp:pw2'}
\end{align}
where
\begin{equation}\label{cM:def}
\begin{split}
\cM_{\mu+\nu}(\cF)(t)=&\sup_{\substack{(s,y)\in\Lambda_0(t),\\|y|\le3\w{t}}}
\{\w{|y|}^\frac32\w{|y|+s}^{1+\mu+\nu}|\cF(s,y)|\}\\
&+\sup_{(s,y)\in\Lambda_1(t)}\{\w{s}^{\frac32+\mu+\nu}\w{|y|-s}|\cF(s,y)|\}.
\end{split}
\end{equation}
\end{lemma}
\begin{proof}
Recall the Poisson formula
\begin{equation*}
\vp_{inh}(t,x)=\frac{1}{2\pi}\int_0^t\int_{|y-x|\le t-s}
\frac{\cF(s,y)dyds}{\sqrt{(t-s)^2-|y-x|^2}}.
\end{equation*}
In the domain $\{(y,s): |y-x|\le t-s\}$, one has $\w{|y|+s}\ls\w{|x|+t}$ and $|y|\le|x|+t\le3\w{t}$.
Therefore, we obtain
\begin{equation*}
\w{|x|+t}^{-\mu_1}|\vp_{inh}(t,x)|\ls\int_0^t\int_{|y-x|\le t-s}
\frac{\w{|y|+s}^{-\mu_1}|\cF(s,y)|dyds}{\sqrt{(t-s)^2-|y-x|^2}}.
\end{equation*}
Applying \eqref{sharp:pw1} to the above integration yields \eqref{sharp:pw1'}.
The proof of \eqref{sharp:pw2'} is analogous.
\end{proof}
Next, we study the pointwise estimates of $\vp_{hom}$.
\begin{lemma}[Estimates of $\vp_{hom}$]\label{lem:weight:hom}
Let $\vp_{hom}$ be defined by \eqref{phi:hom:def}. It holds that
\begin{align}
\w{|x|+t}^\frac12\w{|x|-t}^\frac12|\vp_{hom}(t,x)|
&\ls\|\w{|y|}\vp_0(y)\|_{W_y^{2,1}}+\|\w{|y|}\vp_1(y)\|_{W_y^{1,1}},\label{hom:ineq}\\
\w{|x|+t}^\frac12\w{|x|-t}^\frac32|\nabla\vp_{hom}(t,x)|
&\ls\|\w{|y|}^2\vp_0(y)\|_{W_y^{3,1}}+\|\w{|y|}^2\vp_1(y)\|_{W_y^{2,1}}.\label{hom:grad}
\end{align}
\end{lemma}
\begin{proof}
The inequality \eqref{hom:ineq} is just Lemma 3.2 of \cite{HouYin20jde}.

Next,  we derive \eqref{hom:grad} by \eqref{hom:ineq}. In fact,
for any vector field $\hat Z\in\{\p,\cS,\Omega,t\p_i+x_i\p_t,i=1,2\}$, one has $\Box\hat Z\vp_{hom}=0$.
Applying \eqref{hom:ineq} to $\hat Z\vp_{hom}$ yields
\begin{equation*}
\begin{split}
\w{|x|+t}^\frac12\w{|x|-t}^\frac12|\hat Z\vp_{hom}(t,x)|
&\ls\|\w{|y|}\hat Z\vp_{hom}(0,y)\|_{W_y^{2,1}}
+\|\w{|y|}\p_t\hat Z\vp_{hom}(0,y)\|_{W_y^{1,1}}\\
&\ls\|\w{|y|}^2\vp_0(y)\|_{W_y^{3,1}}+\|\w{|y|}^2\vp_1(y)\|_{W_y^{2,1}}.
\end{split}
\end{equation*}
This, together with
\begin{equation*}
\w{|x|-t}|\nabla\vp_{hom}(t,x)|\ls
\sum_{\hat Z\in\{\p,\cS,\Omega,t\p_i+x_i\p_t,i=1,2\}}|\hat Z\vp_{hom}(t,x)|,
\end{equation*}
derives Lemma~\ref{lem:weight:hom}.
\end{proof}

\subsection{Improved pointwise estimates of $u$}
This subsection is devoted to improve the pointwise estimates of $u$ by the weighted $L^\infty$-$L^\infty$ estimates
in subsection \ref{sect:weight:pw}.
For this purpose, we need to find the related wave equation hidden in the equations \eqref{reducedEuler}.

By the Helmholtz decomposition in subsection \ref{sect:decomp}, there exists a potential function $\phi(t,x)$ such that
\begin{equation}\label{potential:def}
u=P_1u+P_2u=\nabla\phi+P_2u.
\end{equation}
The inherent wave equation of $\phi$ can be directly deduced from \eqref{reducedEuler}, see appendix \ref{appendix:A} for details.
\begin{equation}\label{potential:wave}
\Box\phi=F:=\p_t\cA-(2-\sigma)Q_1-u\cdot Q_2,
\end{equation}
where the nonlocal term $\cA$ is defined by
\begin{equation}\label{cA:def}
\cA:=-(-\Delta)^{-1}\curl(u\curl u),\quad\lim_{|x|\rightarrow\infty}\cA(t,x)=0.
\end{equation}
In addition,
\begin{equation}\label{sigma:potential}
\sigma=-\p_t\phi+\cA-\frac12(|u|^2-\sigma^2),
\end{equation}
which is achieved in appendix \ref{appendix:A}.

Acting $(\cS+2)^{a_s}Z^{a_z}$ on \eqref{potential:wave}, we can get the equation of $\Gamma^a\phi$:
\begin{equation}\label{potential:wave:high}
\Box\Gamma^a\phi=F^a:=\sum_{b\le a}C^a_b\Gamma^b\p_t\cA
+\sum_{b+c\le a}C^a_{bc}Q_1^{bc}+\sum_{b+c+d\le a}C^a_{bcd}
\Big\{Q_1^{bc}\Gamma^d\sigma-Q_{2i}^{bc}\tilde\Gamma^du_i\Big\},
\end{equation}
where $C^a_{***}$ are some suitable constants.

At first, we deal with the pointwise estimates of the nonlocal term $\cA$.
\begin{lemma}[Estimates of $\cA$]\label{lem:cA:pw}
Under bootstrap assumptions \eqref{bootstrap}, for $|a'|\le N_1-2$, it holds that
\begin{equation}\label{cA:pw}
|\Gamma^{a'}\cA(t,x)|\ls\w{|x|}^{-\frac83}M\delta\w{t}^{2M'\ve}
\big\{M\delta+M\ve\w{t}^{-1}\ln(2+t)\big\}.
\end{equation}
\end{lemma}
\begin{proof}
In view of \eqref{pw:prepare5}, \eqref{cA:pw} can be directly achieved by the following $L^p$ estimates
\begin{equation}\label{cA:Lp}
\begin{split}
&\quad\sum_{|a|\le|a'|+1}\Big\{\|\w{|x|}^7\nabla\Gamma^a\cA(t,x)\|_{L^5}
+\sum_{p=\frac{10}{9},\frac{10}{7}}\|\nabla\Gamma^a\cA(t,x)\|_{L^p}\Big\}\\
&\ls M\delta\w{t}^{2M'\ve}\big\{M\delta+M\ve\w{t}^{-1}\ln(2+t)\big\}.
\end{split}
\end{equation}
Next we focus on the proof of \eqref{cA:Lp}.
Note that
\begin{equation}\label{Gamma:cA}
\Gamma^a\cA=-\sum_{b+c\le a}C^a_{bc}(-\Delta)^{-1}
\curl(\tilde\Gamma^bu~\Gamma^c\curl u),
\end{equation}
which can be achieved by applying the following equality repeatedly
\begin{equation*}
\Delta\Gamma\cA=\curl(u~\Gamma\curl u+\tilde\Gamma u\curl u).
\end{equation*}
Indeed, for $\Gamma=\Omega$ (the other case of $\Gamma$ is analogous),
direct computation yields
\begin{equation*}
\begin{split}
\Delta\Omega\cA&=\Omega\Delta\cA=\Omega\curl(u\curl u)=\curl\tilde\Omega(u\curl u)\\
&=\curl\{\Omega(u\curl u)-u^\perp\curl u\}\\
&=\curl\{(\Omega u)\curl u+u\Omega\curl u-u^\perp\curl u\}\\
&=\curl\{(\tilde\Omega u)\curl u+u\Omega\curl u\}.
\end{split}
\end{equation*}
It concludes from the $L^p$ boundedness of the Riesz operators and \eqref{Gamma:cA} that
\begin{equation*}
\begin{split}
\sum_{p=\frac{10}{9},\frac{10}{7}}\|\nabla\Gamma^a\cA(t,x)\|_{L_x^p}
&\ls\sum_{p=\frac{10}{9},\frac{10}{7}}
~\sum_{b+c\le a}\|\tilde\Gamma^bu~\Gamma^c\curl u\|_{L_x^p}\\
&\ls\sum_{p=\frac{10}{9},\frac{10}{7}}
~\sum_{b+c\le a}\|\w{|x|}^{-1}\Gamma^bu\|_{L_x^\infty}
\|\w{|x|}\Gamma^c\curl u\|_{L_x^p}\\
&\ls\bW_{|a|}(t)\big\{\cW_{|a|+1}(t)
+\w{t}^{-1}\ln(2+t)[E_{|a|+2}(t)+\cX_{|a|+2}(t)]\big\}\\
&\ls M\delta\w{t}^{2M'\ve}\big\{M\delta+M\ve\w{t}^{-1}\ln(2+t)\big\},
\end{split}
\end{equation*}
where we have used \eqref{pw:wave3'} and \eqref{pw:wave4}.

Moreover, according to \cite{Stein}, one knows that $\w{|x|}^7$ belongs to $A_5$ class.
Thereafter, we have
\begin{equation*}
\begin{split}
\|\w{|x|}^7\nabla\Gamma^a\cA(t,x)\|_{L_x^5}
&\ls\sum_{b+c\le a}\|\w{|x|}^7\tilde\Gamma^bu~\Gamma^c\curl u\|_{L_x^5}\\
&\ls\sum_{b+c\le a}\|\w{|x|}^{-1}\tilde\Gamma^bu\|_{L_x^\infty}
\|\w{|x|}^8\Gamma^c\curl u\|_{L_x^5}\\
&\ls M\delta\w{t}^{2M'\ve}\big\{M\delta+M\ve\w{t}^{-1}\ln(2+t)\big\}.
\end{split}
\end{equation*}
Thus, we have proved \eqref{cA:Lp}.
\end{proof}

Secondly, we focus on the pointwise estimates of the potential function $\Gamma^a\phi$, where  $\phi$ is defined in \eqref{potential:def}.
\begin{lemma}[Estimates of $\Gamma^a\phi$]\label{lem:potential:pw}
Under bootstrap assumptions \eqref{bootstrap}, for $|a|\le N_1-3$ and $|x|\le2\w{t}$, it holds that
\begin{equation}\label{potential:1}
\w{|x|+t}^\frac14\w{|x|-t}^\frac{1}{16}|\Gamma^a\phi(t,x)|
\ls M\delta\w{t}^\frac12+M\ve.
\end{equation}
\end{lemma}
\begin{proof}
By applying the weighted $L^\infty$-$L^\infty$ estimates \eqref{sharp:pw1'} and \eqref{hom:ineq}
with $\mu_1=\frac14$ and $\nu=\mu=\frac{1}{16}$ to $\Gamma^a\phi$ in \eqref{potential:wave:high}, we obtain
\begin{equation}\label{potential:2}
\begin{split}
&\quad\w{|x|+t}^\frac14\w{|x|-t}^\frac{1}{16}|\Gamma^a\phi(t,x)|\\
&\ls\|\w{|y|}\Gamma^a\phi(0,y)\|_{W_y^{2,1}}+\|\w{|y|}\p_t\Gamma^a\phi(0,y)\|_{W_y^{1,1}}
+\cM_{-\frac18}(F^a)(t).
\end{split}
\end{equation}

According to the definition of $F^a$ in \eqref{potential:wave:high}, we arrive at
\begin{equation}\label{potential:3}
\begin{split}
\cM_{-\frac18}(F^a)(t)\ls\sum_{b\le a}\cM_{-\frac18}(\Gamma^b\p_t\cA)(t)
+\sum_{b+c\le a}\cM_{-\frac18}(Q_1^{bc})(t)\\
+\sum_{b+c+d\le a}\cM_{-\frac18}
(|Q_1^{bc}\Gamma^d\sigma|+|Q_{2i}^{bc}\tilde\Gamma^du_i|)(t).
\end{split}
\end{equation}
Note that it only suffices to deal with the two terms on the right hand side of the first line in \eqref{potential:3},
since the treatment on the cubic nonlinearities in the second line are much easier.

It follows from the definition \eqref{cM:def} of $\cM_{-\frac18}(\cF)(t)$ and \eqref{cA:pw} that
\begin{equation}\label{potential:4}
\begin{split}
\sum_{b\le a}\cM_{-\frac18}(\Gamma^b\p_t\cA)(t)
&\ls\sum_{b\le a}~\sup_{s\le t,y}
\w{|y|}^\frac32\w{|y|+s}^\frac78|\Gamma^b\p_t\cA(s,y)|\\
&\ls M\delta\sup_{s\le t,y}\w{|y|+s}^\frac78\w{|y|}^{-\frac76}\w{s}^{2M'\ve}
\big\{M\delta+M\ve\w{s}^{-1}\ln(2+s)\big\}\\
&\ls M\delta.
\end{split}
\end{equation}
The control of $\cM_{-\frac18}(Q_1^{bc})(t)$ will be divided into three parts corresponding to the domains
$\Lambda_1(t)$, $\cD_1$ and $\cD_2$, where
\begin{equation*}
\begin{split}
&\cD_1:=\Lambda_0(t)\cap\{(s,y): |y|\le2\w{s}/3\},\\
&\cD_2:=\Lambda_0(t)\cap\{(s,y): 4\w{s}/3\le|y|\le3\w{t}\},
\end{split}
\end{equation*}
and the definitions of $\Lambda_0(t)$ and $\Lambda_1(t)$ see \eqref{Lambda:def}.

In $\cD_1$, applying \eqref{pw:wave3} to \eqref{Qbc:def} yields that
\begin{equation}\label{potential:5}
\begin{split}
&\quad\sum_{b+c\le a}\sup_{(s,y)\in\cD_1}
\w{|y|}^\frac32\w{|y|+s}^\frac78|Q_1^{bc}(s,y)|\\
&\ls\sup_{s\le t}\w{s}^{\frac{11}{8}+2M'\ve}
\Big(M\delta+M\ve\w{s}^{-\frac12}\Big)\Big(M\delta+M\ve\w{s}^{-1}\Big)
\ls M\delta\w{t}^\frac12+M\ve.
\end{split}
\end{equation}
In $\cD_2$, by using \eqref{pw:wave3} again to \eqref{Qbc:def}, we deduce that
\begin{equation}\label{potential:6}
\begin{split}
&\quad\sum_{b+c\le a}\sup_{(s,y)\in\cD_2}
\w{|y|}^\frac32\w{|y|+s}^\frac78|Q_1^{bc}(s,y)|\\
&\ls\sup_{\substack{|y|\le3\w{t},\\s\le t}}\w{|y|}^{\frac{11}{8}+2M'\ve}
\Big(M\delta+M\ve\w{|y|}^{-\frac12}\Big)\Big(M\delta+M\ve\w{|y|}^{-1}\Big)
\ls M\delta\w{t}^\frac12+M\ve.
\end{split}
\end{equation}
In $\Lambda_1(t)$, the null condition structure in \eqref{null:Q1} will play a crucial rule (see Section 5 below).
Applying \eqref{pw:wave3} to the terms that containing the factor $\frac{1}{|x|}$ in \eqref{null:Q1}, we then see that
\begin{equation}\label{potential:7}
\begin{split}
&\quad\sum_{b+c\le a}\sup_{(s,y)\in\Lambda_1(t)}
\w{s}^\frac{11}{8}\w{|y|-s}|Q_1^{bc}(s,y)|\\
&\ls\sup_{|y|\le4s/3\le4t/3}\w{s}^\frac38
\Big(M\delta\w{s}^{M'\ve-\frac83}\w{|y|-s}^\frac12+M\ve\w{s}^{M'\ve-\frac12}\Big)^2\\
&\quad+\sum_{b+c\le a}\sup_{(s,y)\in\Lambda_1(t)}\w{s}^\frac{11}{8}\w{|y|-s}
(|\nabla\Gamma^b\sigma\tilde\Gamma^cg|+|\Gamma^b\sigma\nabla\tilde\Gamma^cg|).
\end{split}
\end{equation}
By employing \eqref{pw:wave3} and \eqref{good:pw1} to $\nabla\Gamma^b\sigma\tilde\Gamma^cg$ and $\Gamma^b\sigma\nabla\tilde\Gamma^cg$,
one has
\begin{equation}\label{potential:8}
\begin{split}
\sum_{b+c\le a}\sup_{(s,y)\in\Lambda_1(t)}\w{s}^\frac{11}{8}\w{|y|-s}
(|\nabla\Gamma^b\sigma\tilde\Gamma^cg|+|\Gamma^b\sigma\nabla\tilde\Gamma^cg|)
\ls M\delta+M\ve.
\end{split}
\end{equation}
Collecting \eqref{potential:5}--\eqref{potential:8} together, we eventually arrive at
\begin{equation}\label{potential:9}
\sum_{b+c\le a}\cM_{-\frac18}(Q_1^{bc})(t)\ls M\delta\w{t}^\frac12+M\ve.
\end{equation}
At last, we turn to estimate  the initial data on the right hand side of \eqref{potential:2}.
It is deduced from \eqref{initial:data} and \eqref{sigma:potential} that
\begin{equation}\label{potential:10}
\|\w{|y|}\Gamma^a\phi(0,y)\|_{W_y^{2,1}}
+\|\w{|y|}\p_t\Gamma^a\phi(0,y)\|_{W_y^{1,1}}\ls\ve\ls M\ve.
\end{equation}
Substituting \eqref{potential:3}, \eqref{potential:4}, \eqref{potential:9} and \eqref{potential:10} into \eqref{potential:2}
yields \eqref{potential:1}.
\end{proof}

Next, the pointwise estimates of the good unknown $g$ can be improved.
\begin{lemma}[Improved estimates of $\tilde\Gamma^ag$]\label{lem:good:improve}
Under bootstrap assumptions \eqref{bootstrap}, for $|a|\le N_1-4$ and $\w{t}/8\le|x|\le2\w{t}$, it holds that
\begin{equation}\label{good:improve1}
|\tilde\Gamma^ag(t,x)|\ls M\delta\w{t}^{-\frac34}+M\ve\w{t}^{-\frac54}
+M\ve\w{t}^{M'\ve-\frac32}\w{|x|-t}^\frac12.
\end{equation}
\end{lemma}
\begin{proof}
The key idea to achieve \eqref{good:improve1} is to make full use of the potential function $\phi$ with the relations $u=\nabla\phi+P_2u$ and \eqref{sigma:potential}.
It follows from the definition \eqref{goodunknown:def}, \eqref{null:structure} and tedious computation that
\begin{equation}\label{good:improve2}
\begin{split}
t\tilde\Gamma^ag_i&=(t-|x|)\tilde\Gamma^au_i+|x|\tilde\Gamma^au_i
-t\omega_i\tilde\Gamma^a\sigma
-\frac{t}{|x|}\sum_{b+c\le a}f_i^{a,b}(x)\Gamma^c\sigma\\
&=(t-|x|)\tilde\Gamma^au_i+|x|P_2\tilde\Gamma^au_i+|x|\tilde\Gamma^a\p_i\phi
+t\omega_i\tilde\Gamma^a\p_t\phi-t\omega_i\Gamma^a\cA\\
&\quad-\frac{t}{|x|}\sum_{b+c\le a}f_i^{a,b}(x)\Gamma^c\sigma
+\frac12t\omega_i\Gamma^a(|u|^2-\sigma^2)\\
&=(t-|x|)\tilde\Gamma^au_i+|x|P_2\tilde\Gamma^au_i
+\sum_{b\le a}C^a_b(|x|\p_i+t\omega_i\p_t)\Gamma^b\phi-t\omega_i\Gamma^a\cA\\
&\quad-\frac{t}{|x|}\sum_{b+c\le a}f_i^{a,b}(x)\Gamma^c\sigma
+t\omega_i\sum_{b+c\le a}C^a_{bc}(\tilde\Gamma^bu_j\tilde\Gamma^cu_j
-\Gamma^b\sigma\Gamma^c\sigma).
\end{split}
\end{equation}
By utilizing \eqref{radial:angular} and the first line of \eqref{null:structure} to the last summation in \eqref{good:improve2}, we arrive at
\begin{equation}\label{good:improve3}
\begin{split}
t\tilde\Gamma^ag_i&=(t-|x|)\tilde\Gamma^au_i+|x|P_2\tilde\Gamma^au_i
+\sum_{b\le a}C^a_b(\omega_i\cS\Gamma^b\phi+\Omega\Gamma^b\phi)-t\omega_i\Gamma^a\cA\\
&\quad-\frac{t}{|x|}\sum_{b+c\le a}f_i^{a,b}(x)\Gamma^c\sigma
+t\omega_i\sum_{b+c\le a}C^a_{bc}(\tilde\Gamma^bg_j\tilde\Gamma^cu_j
+\omega_j\Gamma^b\sigma\tilde\Gamma^cg_j)\\
&\quad+\frac{t\omega_i}{|x|}\sum_{b_1+b_2+c\le a}f_j^{b,b_1}(x)
\Gamma^{b_2}\sigma\tilde\Gamma^cu_j
+\frac{t\omega_i}{|x|}\sum_{b+c_1+c_2\le b}f_j^{c,c_1}(x)
\omega_j\Gamma^b\sigma\Gamma^{c_2}\sigma.
\end{split}
\end{equation}
Applying \eqref{pw:P2u}, \eqref{pw:wave3}, \eqref{cA:pw} and \eqref{potential:1} to \eqref{good:improve3} leads to
\begin{equation}\label{good:improve4}
\begin{split}
\w{t}|\tilde\Gamma^ag_i|&\ls M\delta\w{t}^{M'\ve-\frac53}
+M\ve\w{t}^{M'\ve-\frac12}\w{|x|-t}^\frac12
+\w{t}^{-\frac14}\big\{M\delta\w{t}^\frac12+M\ve\big\}\\
&\quad+M\delta\w{t}^{2M'\ve-\frac53}\big\{M\delta+M\ve\w{t}^{-1}\ln(2+t)\big\}
+M\ve\w{t}\sum_{b\le a}|\tilde\Gamma^bg|.
\end{split}
\end{equation}
Therefore, combining \eqref{good:improve4} with the smallness of $M\ve$ derives \eqref{good:improve1}.
\end{proof}

Finally, we turn to the estimates of the velocity $u=\nabla\phi+P_2u$.
\begin{lemma}[Estimates of $\nabla\Gamma^a\phi$]\label{lem:grad:phi}
Under bootstrap assumptions \eqref{bootstrap}, for $|a|\le N_1-4$ and $|x|\le2\w{t}$, it holds that
\begin{equation}\label{grad:phi:pw1}
\w{|x|-t}^\frac98|\nabla\Gamma^a\phi(t,x)|\ls M\delta\w{t}^\frac34+M\ve.
\end{equation}
\end{lemma}

\begin{proof}
Applying \eqref{sharp:pw2'} and \eqref{hom:grad} to \eqref{potential:wave:high} with $\nu=\frac18$ and $\mu=\frac{1}{24}$ yields
\begin{equation}\label{grad:phi:pw2}
\begin{split}
&\quad\w{|x|}^\frac12\w{|x|-t}^\frac98|\Gamma^a\phi(t,x)|\\
&\ls\|\w{|y|}^2\Gamma^a\phi(0,y)\|_{W_y^{3,1}}
+\|\w{|y|}^2\p_t\Gamma^a\phi(0,y)\|_{W_y^{2,1}}+\sum_{|b|\le|a|+1}\cM_{\frac16}(F^b)(t).
\end{split}
\end{equation}

Similarly to Lemma~\ref{lem:potential:pw}, we can obtain
\begin{equation}\label{grad:phi:pw3}
\begin{split}
\sum_{|b|\le|a|+1}\cM_{\frac16}(F^b)(t)
\ls\sum_{|b|\le|a|+1}\cM_{\frac16}(\Gamma^b\p_t\cA)(t)
+\sum_{|b|+|c|\le|a|+1}\cM_{\frac16}(Q_1^{bc})(t)\\
+\sum_{|b|+|c|+|d|\le|a|+1}\cM_{\frac16}
(|Q_1^{bc}\Gamma^d\sigma|+|Q_{2i}^{bc}\tilde\Gamma^du_i|)(t).
\end{split}
\end{equation}
It follows from the definition \eqref{cM:def} of $\cM_{\frac16}(\cF)(t)$ and \eqref{cA:pw} that
\begin{equation}\label{grad:phi:pw4}
\begin{split}
\sum_{|b|\le|a|+1}\cM_{\frac16}(\Gamma^b\p_t\cA)(t)
&\ls\sum_{|b|\le|a|+1}~\sup_{s\le t,y}
\w{|y|}^\frac32\w{|y|+s}^\frac76|\Gamma^b\p_t\cA(s,y)|\\
&\ls M\delta\sup_{s\le t,y}\w{|y|+s}^\frac76\w{|y|}^{-\frac76}\w{s}^{2M'\ve}
\big\{M\delta+M\ve\w{s}^{-1}\ln(2+s)\big\}\\
&\ls M\delta\w{t}^{\frac16+2M'\ve}\ln(2+t)\ls M\delta\w{t}^\frac15.
\end{split}
\end{equation}
In $\cD_1$, applying \eqref{pw:wave4} and \eqref{pw:wave6} to \eqref{Qbc:def} implies
\begin{equation}\label{grad:phi:pw5}
\begin{split}
&\quad\sum_{b+c\le a}\sup_{(s,y)\in\cD_1}\w{|y|}^\frac32\w{|y|+s}^\frac76|Q_1^{bc}(s,y)|\\
&\ls\sup_{s\le t}~\w{s}^\frac53
\Big(M\delta\w{s}^{M'\ve}+M\ve\w{s}^{2M'\ve-1}\ln(2+s)\Big)^2\\
&\ls M\delta\w{t}^\frac34+M\ve.
\end{split}
\end{equation}
In $\cD_2$, by using \eqref{pw:wave3}, \eqref{pw:wave5} to \eqref{Qbc:def}, we get
\begin{equation}\label{grad:phi:pw6}
\begin{split}
&\quad\sum_{b+c\le a}\sup_{(s,y)\in\cD_2}
\w{|y|}^\frac32\w{|y|+s}^\frac76|Q_1^{bc}(s,y)|\\
&\ls\sup_{\substack{|y|\le3\w{t},\\s\le t}}\w{|y|}^\frac53
\Big(M\delta\w{s}^{M'\ve}+M\ve\w{|y|}^{2M'\ve-1}\Big)^2\\
&\ls M\delta\w{t}^\frac34+M\ve.
\end{split}
\end{equation}

Similarly to Lemma \ref{lem:potential:pw}, the treatment in $\Lambda_1(t)$ will also need to
make full use of the null condition structure in \eqref{null:Q1}. Indeed, applying \eqref{pw:wave3} to
the terms that containing the factor $\frac{1}{|x|}$ in \eqref{null:Q1}, we then have
\begin{equation}\label{grad:phi:pw7}
\begin{split}
&\quad\sum_{b+c\le a}\sup_{(s,y)\in\Lambda_1(t)}\w{s}^\frac53\w{|y|-s}|Q_1^{bc}(s,y)|\\
&\ls\sup_{|y|\le4s/3\le4t/3}\w{s}^\frac23\Big(M\delta\w{s}^{M'\ve}\w{|y|-s}^\frac12
+M\ve\w{s}^{M'\ve-\frac12}\Big)^2\\
&\quad+\sum_{b+c\le a}\sup_{(s,y)\in\Lambda_1(t)}\w{s}^\frac53\w{|y|-s}
\Big\{|\nabla\Gamma^b\sigma\tilde\Gamma^cg|+|\Gamma^b\sigma\nabla\tilde\Gamma^cg|\Big\}.
\end{split}
\end{equation}
On the other hand, by using \eqref{pw:wave5} and \eqref{good:improve1} to $\nabla\Gamma^b\sigma\tilde\Gamma^cg$, one has
\begin{equation}\label{grad:phi:pw8}
\begin{split}
&\quad\sum_{b+c\le a}\sup_{(s,y)\in\Lambda_1(t)}\w{s}^\frac53\w{|y|-s}
|\nabla\Gamma^b\sigma\tilde\Gamma^cg|\\
&\ls\sup_{s\le t}\Big\{M\delta\w{s}^{2M'\ve+\frac53-\frac12-\frac34}
+M\ve\w{s}^{2M'\ve+\frac53-\frac12-\frac54}\Big\}\\
&\ls M\delta\w{t}^\frac34+M\ve.
\end{split}
\end{equation}
In addition, employing \eqref{pw:wave1}, \eqref{good:pw2} to $\Gamma^b\sigma\nabla\tilde\Gamma^cg$ yields that
\begin{equation}\label{grad:phi:pw9}
\sum_{b+c\le a}\sup_{(s,y)\in\Lambda_1(t)}\w{s}^\frac53\w{|y|-s}
|\Gamma^b\sigma\nabla\tilde\Gamma^cg|
\ls M\ve.
\end{equation}
Collecting \eqref{grad:phi:pw5}--\eqref{grad:phi:pw9} together, we eventually achieve
\begin{equation}\label{grad:phi:pw10}
\sum_{b+c\le a}\cM_{\frac16}(Q_1^{bc})(t)\ls M\delta\w{t}^\frac34+M\ve.
\end{equation}
Note that similarly to \eqref{potential:10}, we have
\begin{equation}\label{grad:phi:pw11}
\|\w{|y|}^2\Gamma^a\phi(0,y)\|_{W_y^{3,1}}
+\|\w{|y|}^2\p_t\Gamma^a\phi(0,y)\|_{W_y^{2,1}}\ls M\ve.
\end{equation}
Therefore, substituting  \eqref{grad:phi:pw3}, \eqref{grad:phi:pw4}, \eqref{grad:phi:pw10} and \eqref{grad:phi:pw11} into \eqref{grad:phi:pw2}
derives \eqref{grad:phi:pw1}.
\end{proof}

Combining \eqref{grad:phi:pw1} and \eqref{pw:P2u} with the decomposition $\tilde\Gamma^au=\nabla\Gamma^a\phi+P_2\tilde\Gamma^au$,
 we arrive at
\begin{corollary}\label{coro:improve}
Under bootstrap assumptions \eqref{bootstrap}, for $|a|\le N_1-4$ and $|x|\le2\w{t}$, it holds that
\begin{equation}\label{velocity:pw}
|\tilde\Gamma^au(t,x)|\ls M\delta\w{|x|}^{-\frac83}
+\w{|x|-t}^{-\frac98}\big\{M\delta\w{t}^\frac34+M\ve\big\}.
\end{equation}
\end{corollary}

\section{Energy estimates}\label{sect5}

\subsection{Elementary energy estimates}\label{sect5:1}
This subsection is aimed to establish the elementary energy estimates for $E_m(t)$, which
is defined by \eqref{energy:def}.
\begin{lemma}\label{lem:energy:wave}
Under bootstrap assumptions \eqref{bootstrap}, we have that
\begin{align}
E^2_N(t')&\ls E^2_N(0)+\int_0^{t'}E^2_N(t)\big\{M\delta+M\ve\w{t}^{-1}\big\}dt
\nonumber\\
&\quad+\int_0^{t'}E_N(t)M\delta\w{t}^{2M'\ve}
\big\{M\delta+M\ve\w{t}^{-\frac78}\big\}dt,\label{energy:wave:high}\\
E^2_{N_1-4}(t')&\ls E^2_{N_1-4}(0)+\int_0^{t'}E^2_{N_1-4}(t)
\big\{M\delta+M\ve\w{t}^{-\frac98}\big\}dt \nonumber\\
&\quad+\int_0^{t'}E_{N_1-4}(t)M\delta\w{t}^{2M'\ve}
\big\{M\delta+M\ve\w{t}^{-\frac78}\big\}dt.\label{energy:wave:low}
\end{align}
\end{lemma}
\begin{proof}
For the multi-index $a$ with $|a|=m\le N$, multiplying \eqref{high:eqn} by $2e^q\Gamma^a\sigma$
and $2e^q\tilde\Gamma^au$, respectively, yields the following equality
\begin{equation*}
\begin{split}
&\p_t\{e^q(|\Gamma^a\sigma|^2+|\tilde\Gamma^au|^2)\}
+2\dive\{e^q(1-\sigma)\Gamma^a\sigma\tilde\Gamma^au\}
+\dive\{e^qu(|\Gamma^a\sigma|^2+|\tilde\Gamma^au|^2)\}\\
&+\frac{e^q}{\w{|x|-t}^\frac32}\sum_{i=1}^2\Big\{|\tilde\Gamma^au_i-\omega_i\Gamma^a\sigma|^2
-u_i\omega_i(|\Gamma^a\sigma|^2+|\tilde\Gamma^au|^2)
+2\sigma\omega_i\Gamma^a\sigma\tilde\Gamma^au_i\Big\}\\
&=e^q(|\Gamma^a\sigma|^2+|\tilde\Gamma^au|^2)\dive u
-2e^q\Gamma^a\sigma\tilde\Gamma^au\cdot\nabla\sigma
+\sum_{\substack{b+c=a,\\c<a}}2e^qC^a_{bc}(Q_1^{bc}\Gamma^a\sigma
+Q_2^{bc}\cdot\tilde\Gamma^au),
\end{split}
\end{equation*}
where the ghost weight $e^q=e^{q(|x|-t)}$ with $\ds q(|x|-t)=\int_{-\infty}^{|x|-t}\w{s}^{-\frac32}ds$ was introduced in \cite{Alinhac01}
for treating the global small solution problem of the second order 2D quasilinear wave equation.
Integrating the above equality over $[0,t']\times\R^2$ derives that
\begin{equation}\label{energy:wave:ineq}
\begin{split}
&\quad E^2_{|a|}(t')+\sum_{i=1}^2\int_0^{t'}\int\frac{1}{\w{|x|-t}^\frac32}
|\tilde\Gamma^au_i-\omega_i\Gamma^a\sigma|^2dxdt\\
&\ls E^2_{|a|}(0)+\int_0^{t'}\int\Big\{|I^a|+\sum_{\substack{b+c=a,\\c<a}}
(|Q_1^{bc}\Gamma^a\sigma|+|Q_2^{bc}\cdot\tilde\Gamma^au|)\Big\}dxdt,
\end{split}
\end{equation}
where
\begin{equation}\label{Ia:def}
\begin{split}
I^a:=&~(|\Gamma^a\sigma|^2+|\tilde\Gamma^au|^2)\dive u
-2\Gamma^a\sigma\tilde\Gamma^au\cdot\nabla\sigma\\
&+\frac{1}{\w{|x|-t}^\frac32}\sum_{i=1}^2
\Big\{u_i\omega_i(|\Gamma^a\sigma|^2+|\tilde\Gamma^au|^2)
-2\sigma\omega_i\Gamma^a\sigma\tilde\Gamma^au_i\Big\}.
\end{split}
\end{equation}
As in Lemma~\ref{lem:H1norm}, the space domain is still divided into two parts of $|x|\le\w{t}/8$ and $|x|\ge\w{t}/8$.

Firstly, we deal with $I^a$, $Q_1^{bc}$ and $Q_2^{bc}$ in the region $|x|\le\w{t}/8$.
For $I^a$, it concludes from \eqref{pw:wave3} and \eqref{pw:wave6} that
\begin{equation}\label{Ia:ineq1}
\int_{|x|\le\w{t}/8}|I^a|dx
\ls E^2_m(t)\big\{M\ve\w{t}^{-\frac32}+M\delta\big\}.
\end{equation}

In view of $|b|+|c|\le N\le2N_1-11$, then $|b|\le N_1-5$ or $|c|\le N_1-7$ holds.

If $|b|\le N_1-1$, applying \eqref{div:curl:ineq} to $\nabla\tilde\Gamma^cu$ and
applying \eqref{pw:wave4} to $\Gamma^b\sigma$, $\tilde\Gamma^bu$, respectively, infer that
\begin{equation}\label{Qbc:ineq1}
\begin{split}
&\quad\int_{|x|\le\w{t}/8}(|Q_1^{bc}\Gamma^a\sigma|+|Q_2^{bc}\cdot\tilde\Gamma^au|)dx\\
&\ls E_m(t)\big\{\cX_m(t)+E_{m-1}(t)+\w{t}\cW_{m-1}(t)\big\}
\big\{M\ve\w{t}^{M'\ve-2}\ln(2+t)+M\delta\w{t}^{M'\ve-1}\big\}\\
&\ls E_m(t)\big\{E_m(t)+\w{t}\cW_{m-1}(t)\big\}
\big\{M\ve\w{t}^{M'\ve-2}\ln(2+t)+M\delta\w{t}^{M'\ve-1}\big\}\\
&\ls E^2_m(t)\big\{M\ve\w{t}^{-\frac32}+M\delta\w{t}^{M'\ve-1}\big\}
+E_m(t)M\delta\w{t}^{2M'\ve}\big\{M\delta+M\ve\w{t}^{-\frac78}\big\},
\end{split}
\end{equation}
where we have also used \eqref{bootstrap} and \eqref{H1norm}.

If $|c|\le N_1-5$, by using \eqref{pw:wave6} and \eqref{velocity:pw} to $\nabla\Gamma^c\sigma$ and $\nabla\tilde\Gamma^cu$, respectively, we can obtain
\begin{equation}\label{Qbc:ineq2}
\int_{|x|\le\w{t}/8}(|Q_1^{bc}\Gamma^a\sigma|+|Q_2^{bc}\cdot\tilde\Gamma^au|)dx
\ls E^2_m(t)\big\{M\ve\w{t}^{-\frac98}+M\delta\big\}.
\end{equation}

Now, we turn to the treatments of $I^a$, $Q_1^{bc}$ and $Q_2^{bc}$ in the region $|x|\ge\w{t}/8$,
For $I^a$ defined by \eqref{Ia:def}, it follows from the definition of the good unknown \eqref{goodunknown:def} that
\begin{equation}\label{Ia:ineq2}
\begin{split}
I^a=&\sum_{i=1}^2\Big\{\dive u|\tilde\Gamma^au_i-\omega_i\Gamma^a\sigma|^2
+2\Gamma^a\sigma\tilde\Gamma^au_i(\omega_i\dive u-\p_i\sigma)\Big\}\\
&+\frac{1}{\w{|x|-t}^\frac32}\sum_{i,j=1}^2
\Big\{u_i\omega_i|\tilde\Gamma^au_j-\omega_j\Gamma^a\sigma|^2
+2g_i\omega_i\omega_j\Gamma^a\sigma\tilde\Gamma^au_j\Big\}.
\end{split}
\end{equation}
It follows from \eqref{radial:angular} and the second equality of \eqref{null:structure} that
\begin{equation}\label{Ia:ineq3}
\omega_i\dive u-\p_i\sigma=\omega_i\Big\{\p_jg_j
+\frac{1}{|x|}[f_j^0(x)\p_j\sigma+f_{jj}^0(x)\sigma]\Big\}+\frac{1}{|x|}\Omega\sigma.
\end{equation}
By applying \eqref{pw:wave3} and \eqref{good:pw1} to \eqref{Ia:ineq2} and \eqref{Ia:ineq3}, we see that
\begin{equation}\label{Ia:ineq4}
\int_{|x|\ge\w{t}/8}|I^a|dx\ls\sum_{i=1}^2\int\frac{M\ve}{\w{|x|-t}^\frac32}
|\tilde\Gamma^au_i-\omega_i\Gamma^a\sigma|^2dx
+E^2_m(t)\big\{M\delta+M\ve\w{t}^{-1}\big\},
\end{equation}
where we have used the Cauchy-Schwartz inequality.

Next, we focus on the treatments of $Q_1^{bc}$ and $Q_2^{bc}$ in the region $|x|\ge\w{t}/8$.

\noindent\underline{The case of $|c|\le N_1-7$:}
By the second equality of \eqref{null:structure}, we get
\begin{equation}\label{Qbc:ineq3}
\begin{split}
Q_{2i}^{bc}&=-\tilde\Gamma^bu_j\p_j\tilde\Gamma^cu_i
+\Gamma^b\sigma\p_i\Gamma^c\sigma\\
&=-\p_j\tilde\Gamma^cu_i(\tilde\Gamma^bu_j-\omega_j\Gamma^b\sigma)
+\Gamma^b\sigma(\p_i\Gamma^c\sigma-\omega_j\p_j\tilde\Gamma^cu_i)\\
&=-\p_j\tilde\Gamma^cu_i(\tilde\Gamma^bu_j-\omega_j\Gamma^b\sigma)
+\frac{1}{|x|}\Gamma^b\sigma\Omega\Gamma^c\sigma\\
&\quad-\omega_j\Gamma^b\sigma\Big\{\p_j\tilde\Gamma^cg_i
+\frac{1}{|x|}\sum_{c_1+c_2\le c}[f_i^{c_1}(x)\p_j\Gamma^{c_2}\sigma
+f_{ij}^{c_1}(x)\Gamma^{c_2}\sigma]\Big\},
\end{split}
\end{equation}
and
\begin{equation}\label{Qbc:ineq4}
\begin{split}
Q_1^{bc}&=\Gamma^b\sigma\Big\{\p_i\tilde\Gamma^cg_i
+\frac{1}{|x|}\sum_{c_1+c_2\le c}[f_i^{c_1}(x)\p_i\Gamma^{c_2}\sigma
+f_{ii}^{c_1}(x)\Gamma^{c_2}\sigma]\Big\}\\
&\quad-\p_i\Gamma^c\sigma(\tilde\Gamma^bu_i-\omega_i\Gamma^b\sigma).
\end{split}
\end{equation}
According to \eqref{pw:wave3} with $|c|\le N_1-7$, we arrive at
\begin{equation*}
\begin{split}
|\nabla\Gamma^c\sigma(t,x)|+|\nabla\tilde\Gamma^cu(t,x)|
&\ls\w{t}^{-7}\bW_{|a|+1}(t)
+\w{t}^{-\frac12}\w{|x|-t}^{-1}\big\{E_{|a|+3}(t)+\cX_{|a|+3}(t)\big\}\\
&\ls M\delta+M\ve\w{t}^{-\frac12}\w{|x|-t}^{-1}.
\end{split}
\end{equation*}
Substituting this inequality and \eqref{good:pw1} into \eqref{Qbc:ineq3} and \eqref{Qbc:ineq4} leads to
\begin{equation}\label{Qbc:ineq5}
\begin{split}
&\quad\sum_{\substack{b+c=a,\\|c|\le N_1-7}}\int_{|x|\ge\w{t}/8}
(|Q_1^{bc}\Gamma^a\sigma|+|Q_2^{bc}\cdot\tilde\Gamma^au|)dx\\
&\ls\sum_{|b|\le|a|}\sum_{i=1}^2\int\frac{M\ve}{\w{|x|-t}^\frac32}
|\tilde\Gamma^bu_i-\omega_i\Gamma^b\sigma|^2dx
+E^2_m(t)\big\{M\delta+M\ve\w{t}^{-1}\big\}.
\end{split}
\end{equation}

\noindent\underline{The case of $|b|\le N_1-4$ and the lower order energy:}
By applying \eqref{goodL2norm2}, \eqref{H1norm}, \eqref{good:improve1} to \eqref{null:Q1} and \eqref{null:Q2}, we obtain
\begin{equation}\label{Qbc:ineq6}
\begin{split}
&\sum_{\substack{b+c=a,\\c<a,|b|\le N_1-4}}\int_{\frac{\w{t}}{8}\le|x|\le2\w{t}}
(|Q_1^{bc}\Gamma^a\sigma|+|Q_2^{bc}\cdot\tilde\Gamma^au|)dx\\
&\ls E_m(t)\big\{E_m(t)+\cW_{m-1}(t)\big\}\big\{M\delta+M\ve\w{t}^{-\frac54}\big\}\\
&\ls E_m(t)\big\{E_m(t)+M\delta\w{t}^{M'\ve}\big\}
\big\{M\delta+M\ve\w{t}^{-\frac54}\big\},
\end{split}
\end{equation}
Next, we pay our attention to the related treatments in the region $|x|\ge2\w{t}$.
For this purpose, we need only to deal with $\|\nabla\tilde\Gamma^c u\|_{L^2(|x|\ge2\w{t})}$.
Choose the cutoff function $\tilde\chi(s)\in C^\infty$ which taking values in $[0,1]$ and satisfying
\begin{equation*}
\tilde\chi(s)=\left\{
\begin{aligned}
&1,\qquad\quad s\ge2,\\
&0,\qquad\quad s\le3/2.
\end{aligned}
\right.
\end{equation*}
Thereafter, it is easy to check that
\begin{equation}\label{Qbc:ineq7}
\w{t}\|\nabla\tilde\Gamma^c u\|_{L^2(|x|\ge2\w{t})}\ls E_{|c|}(t)+
\Big\|\w{|x|-t}\nabla\Big(\tilde\chi\Big(\frac{|x|}{t}\Big)
\tilde\Gamma^c u\Big)\Big\|_{L^2}.
\end{equation}
By using \eqref{div:curl:ineq} to the last term on the right hand side of \eqref{Qbc:ineq7}, we then get
\begin{equation}\label{Qbc:ineq8}
\begin{split}
\Big\|\w{|x|-t}\nabla\Big(\tilde\chi\Big(\frac{|x|}{t}\Big)
\tilde\Gamma^c u\Big)\Big\|_{L^2}\ls E_{|c|}(t)+\cX_{|c|+1}(t)+\cW_{|c|}(t),
\end{split}
\end{equation}
Consequently, it concludes from \eqref{pw:wave3}, \eqref{H1norm}, \eqref{Qbc:ineq7} and \eqref{Qbc:ineq8} that
\begin{equation}\label{Qbc:ineq9}
\begin{split}
&\quad\sum_{\substack{b+c=a,\\c<a,|b|\le N_1-4}}\int_{|x|\ge2\w{t}}
(|Q_1^{bc}\Gamma^a\sigma|+|Q_2^{bc}\cdot\tilde\Gamma^au|)dx\\
&\ls E_m(t)\w{t}^{M'\ve-1}\big\{M\delta\w{t}^{-\frac83}+M\ve\w{t}^{-1}\big\}
\big\{E_m(t)+\cX_m(t)+M\delta\w{t}^{M'\ve}\big\}\\
&\ls E^2_m(t)\big\{M\delta+M\ve\w{t}^{-\frac32}\big\}
+E_m(t)M\delta\w{t}^{2M'\ve-1}\big\{M\delta\w{t}^{-\frac83}+M\ve\w{t}^{-1}\big\}\\
&\ls E^2_m(t)\big\{M\delta+M\ve\w{t}^{-\frac32}\big\}
+E_m(t)M^2\ve\delta\w{t}^{-1}.
\end{split}
\end{equation}
For all $m\le N$, substituting \eqref{Ia:ineq1}-\eqref{Qbc:ineq2}, \eqref{Ia:ineq4}, \eqref{Qbc:ineq5},
\eqref{Qbc:ineq6} and \eqref{Qbc:ineq9} into \eqref{energy:wave:ineq} yields \eqref{energy:wave:high}.

Finally, we are dedicated to the lower energy estimates of $I^a$ in the region $|x|\ge\w{t}/8$ with $|a|=m\le N_1-4$.
In the region $|x|\ge2\w{t}$, applying \eqref{pw:wave3} and \eqref{good:pw1} to \eqref{Ia:ineq2} shows that
\begin{equation}\label{Ia:ineq5}
\int_{|x|\ge2\w{t}}|I^a|dx\ls E^2_{N_1-4}(t)
\big\{M\delta+M\ve\w{t}^{-\frac32}\big\}.
\end{equation}
Substituting the first equality in \eqref{null:structure} into \eqref{Ia:ineq2} infers
\begin{equation}\label{Ia:ineq6}
\begin{split}
I^a=&\sum_{i=1}^2\Big\{\dive u\Big|\tilde\Gamma^ag_i
+\frac{1}{|x|}\sum_{a_1+a_2\le a}f_i^{a,a_1}(x)\Gamma^{a_2}\sigma\Big|^2\Big\}\\
&+2\sum_{i,j=1}^2\omega_i\Gamma^a\sigma\tilde\Gamma^au_i
\Big\{\p_jg_j+\frac{1}{|x|}[f_j^0(x)\p_j\sigma+f_{jj}^0(x)\sigma]
+\frac{1}{|x|}\Omega\sigma\Big\}\\
&+\frac{1}{\w{|x|-t}^\frac32}\sum_{i,j=1}^2
\Big\{u_i\omega_i\Big|\tilde\Gamma^ag_j
+\frac{1}{|x|}\sum_{a_1+a_2\le a}f_j^{a,a_1}(x)\Gamma^{a_2}\sigma\Big|^2
+2g_i\omega_i\omega_j\Gamma^a\sigma\tilde\Gamma^au_j\Big\}.
\end{split}
\end{equation}
By plugging \eqref{pw:wave3'}, \eqref{good:pw1} and the improved poinwise estimates
\eqref{good:improve1} of $\tilde\Gamma^ag$ with $|a|\le N_1-4$ into \eqref{Ia:ineq6}, we arrive at
\begin{equation}\label{Ia:ineq7}
\begin{split}
\int_{\frac{\w{t}}{8}\le|x|\le2\w{t}}|I^a|dx&\ls E_{N_1-4}(t)
\big\{E_{N_1-4}(t)+M\delta\big\}\big\{M\delta+M\ve\w{t}^{-\frac54}\big\}\\
&\quad+E^2_{N_1-4}(t)\big\{M\delta+M\ve\w{t}^{-\frac32}\big\}.
\end{split}
\end{equation}
For all $m\le N_1-4$, substituting \eqref{Ia:ineq1}, \eqref{Qbc:ineq1}, \eqref{Qbc:ineq6},
\eqref{Qbc:ineq9}, \eqref{Ia:ineq5} and \eqref{Ia:ineq7} into \eqref{energy:wave:ineq} yields \eqref{energy:wave:low}.
\end{proof}

\subsection{Energy estimates of the vorticity}\label{sect5:2}
Before taking the estimates of the vorticity, we will establish some useful lemmas.
Recalling the definition of the specific vorticity \eqref{curl:def}, then it is easy
to check that
\begin{equation}\label{curl:eqn}
(\p_t+u\cdot\nabla)w=0.
\end{equation}
\begin{lemma}\label{lem:grad:curl}
Under bootstrap assumptions \eqref{bootstrap}, for $m\le N-2$ and $k\le N_1-4$, it holds that
\begin{equation}\label{grad:curl}
\sum_{|a|\le m}\w{|x|}|\nabla\Gamma^aw(t,x)|\ls\sum_{|a'|\le m+1}|\Gamma^{a'}w(t,x)|
+t\w{|x|}^{-\frac32}W_{N_1-4}(t)\sum_{|b|\le m}|\tilde\Gamma^bu(t,x)|,
\end{equation}
and
\begin{equation}\label{grad:curl'}
\sum_{|a|\le k}\w{|x|}|\nabla\Gamma^aw(t,x)|\ls\sum_{|a'|\le k+1}|\Gamma^{a'}w(t,x)|.
\end{equation}
Furthermore, for $|x|\ge\w{t}/8$, it holds that
\begin{equation}\label{grad:curl:cone}
\sum_{|a|\le m}\w{|x|}|\nabla\Gamma^aw(t,x)|\ls\sum_{|a'|\le m+1}|\Gamma^{a'}w(t,x)|.
\end{equation}
\end{lemma}

\begin{proof}
Similarly to the derivation of \eqref{high:eqn}, acting $(\cS+1)^{a_s}Z^{a_z}$ on \eqref{curl:eqn} derives
\begin{equation}\label{curl:eqn:high}
(\p_t+u\cdot\nabla)\Gamma^aw=\sum_{\substack{b+c=a,\\c<a}}C^a_{bc}J_{bc}
:=-\sum_{\substack{b+c=a,\\c<a}}C^a_{bc}\tilde\Gamma^bu\cdot\nabla\Gamma^cw,
\end{equation}
where $\Gamma^a=\cS^{a_s}Z^{a_z}$.
Symbolically, we can see that
\begin{equation}\label{grad:curl:1}
\begin{split}
\w{|x|}\nabla\Gamma^aw&=\sum_{\hat\Gamma\in\{\nabla,\Omega\}}\hat\Gamma\Gamma^aw
+r\p_r\Gamma^aw
=\sum_{\hat\Gamma\in\{\nabla,\Omega,\cS\}}\hat\Gamma\Gamma^aw-t\p_t\Gamma^aw\\
&=\Gamma^{\le1}\Gamma^aw-t\sum_{b+c=a}C^a_{bc}\tilde\Gamma^bu\cdot\nabla\Gamma^cw,
\end{split}
\end{equation}
By using \eqref{pw:curl} to $\nabla\Gamma^cw$ with $|c|\le N_1-7$, we arrive at
\begin{equation}\label{grad:curl:2}
\sum_{\substack{b+c=a,\\|c|\le N_1-7}}|\tilde\Gamma^bu\cdot\nabla\Gamma^cw|
\ls\w{|x|}^{-\frac32}W_{N_1-4}(t)\sum_{|b|\le m}|\tilde\Gamma^bu|.
\end{equation}
If $|c|\ge N_1-6$, then $|b|\le N_1-4$.
It follows from \eqref{bootstrap}, \eqref{pw:wave3} and \eqref{velocity:pw} that
\begin{equation}\label{grad:curl:3}
\sum_{|b|\le N_1-4}t|\tilde\Gamma^bu\cdot\nabla\Gamma^cw|
\ls\sum_{|b|\le N_1-4}\frac{t|\tilde\Gamma^bu|}{\w{|x|}}|\w{|x|}\nabla\Gamma^cw|
\ls M\ve|\w{|x|}\nabla\Gamma^cw|.
\end{equation}
Thereafter, substituting \eqref{grad:curl:2} and \eqref{grad:curl:3} into \eqref{grad:curl:1}
with the smallness of $M\ve$ implies \eqref{grad:curl}.

By plugging \eqref{grad:curl:3} into \eqref{grad:curl:1}, we can achieve \eqref{grad:curl'}.

For the proof of \eqref{grad:curl:cone}, by using \eqref{pw:wave3'} to $\tilde\Gamma^bu$ directly, yields
\begin{equation*}
\begin{split}
t|\tilde\Gamma^bu\cdot\nabla\Gamma^cw|
&\ls|\tilde\Gamma^bu||\w{|x|}\nabla\Gamma^cw|\\
&\ls\w{t}^{-\frac12}|\w{|x|}\nabla\Gamma^cw|
\big\{E_{|b|+2}(t)+\cX_{|b|+2}(t)+\cW_{|b|+1}(t)\big\}\\
&\ls M\ve|\w{|x|}\nabla\Gamma^cw|.
\end{split}
\end{equation*}
This, together with \eqref{grad:curl:1}, implies \eqref{grad:curl:cone}.
\end{proof}

\begin{lemma}\label{lem:curl:equiv}
Under bootstrap assumptions \eqref{bootstrap}, we have
\begin{equation}\label{curl:equiv}
\cW_{N_1-4}(t)\ls W_{N_1-4}(t),\qquad\cW_{N-1}(t)\ls W_{N-1}(t).
\end{equation}
\end{lemma}
\begin{proof}
Note that
\begin{equation*}
\Gamma^a\curl u=\Gamma^aw+\sum_{b+c=a}C^a_{bc}\Gamma^b\sigma\Gamma^c\curl u.
\end{equation*}
Multiplying the above equality by $\w{|x|}$ and taking $L^2$ norm lead to
\begin{equation}\label{curl:equiv:1}
\begin{split}
\|\w{|x|}\Gamma^a\curl u\|_{L^2}&\ls\|\w{|x|}\Gamma^aw\|_{L^2}
+\sum_{\substack{b+c=a,\\|b|\le N-2}}
\|\Gamma^b\sigma\|_{L^\infty}\|\w{|x|}\Gamma^c\curl u\|_{L^2}\\
&\qquad+\sum_{\substack{b+c=a,\\|b|\ge N-1}}
\|\w{|x|}\Gamma^b\sigma\Gamma^c\curl u\|_{L^2}.
\end{split}
\end{equation}
In addition, it follows from \eqref{pw:wave1} that
\begin{equation}\label{curl:equiv:2}
\begin{split}
\sum_{\substack{b+c=a,\\|b|\le N-2}}
\|\Gamma^b\sigma\|_{L^\infty}\|\w{|x|}\Gamma^c\curl u\|_{L^2}
\ls\sum_{c\le a}\|\w{|x|}\Gamma^c\curl u\|_{L^2}.
\end{split}
\end{equation}
For all $|a|\le N_1-4$, substituting \eqref{curl:equiv:2} into \eqref{curl:equiv:1} yields the first inequality of \eqref{curl:equiv}.

Next, we deal the second line of \eqref{curl:equiv:1}.
It is easy to check that
\begin{equation}\label{curl:equiv:3}
\begin{split}
\|\w{|x|}\Gamma^b\sigma\Gamma^c\curl u\|_{L^2}
&\ls\w{t}^{-\frac12}\|\Gamma^b\sigma\|_{L^2(|x|\ge\frac{\w{t}}{2})}
\|\w{|x|}^\frac32\Gamma^c\curl u\|_{L^\infty(|x|\ge\frac{\w{t}}{2})}\\
&\quad+\w{t}^{-\frac14}\Big\|\frac{\w{|x|-t}\Gamma^b\sigma}{\w{|x|}^\frac54}\Big\|_{L^2(|x|\le\frac{\w{t}}{2})}
\|\w{|x|}^\frac32\Gamma^c\curl u\|_{L^\infty(|x|\le\frac{\w{t}}{2})}.
\end{split}
\end{equation}
Applying the Hardy inequality infers
\begin{equation}\label{curl:equiv:4}
\begin{split}
\Big\|\frac{\w{|x|-t}\Gamma^b\sigma}{\w{|x|}^\frac54}\Big\|_{L^2}
\ls\Big\|\frac{\w{|x|-t}\Gamma^b\sigma}{|x|\ln|x|}\Big\|_{L^2}
\ls E_{|b|}(t)+\cX_{|b|+1}(t)\ls M\ve\w{t}^{M'\ve}.
\end{split}
\end{equation}
From $|b|+|c|\le N-1$ and $|b|\ge N-1$, then $|c|\le N-3$ holds.
By plugging \eqref{pw:curl} and \eqref{curl:equiv:4} into \eqref{curl:equiv:3}, we derive
\begin{equation}\label{curl:equiv:5}
\begin{split}
\sum_{\substack{b+c=a,\\|c|\le N-3}}
\|\w{|x|}\Gamma^b\sigma\Gamma^c\curl u\|_{L^2}
\ls M\ve\sum_{|c|\le N-3}\cW_{|c|+2}(t)\ls M\ve\cW_{N-1}(t).
\end{split}
\end{equation}
For all $|a|\le N-1$, combining \eqref{curl:equiv:1}, \eqref{curl:equiv:2}, \eqref{curl:equiv:5} with the smallness of $M\ve$
yields the second inequality of \eqref{curl:equiv}.
\end{proof}

With these lemmas, we begin to take the estimates of the vorticity.
\begin{lemma}[$L^2$ estimates]\label{lem:curl:L2}
Under bootstrap assumptions \eqref{bootstrap}, for $W_m(t)$ defined by \eqref{energy:def}, it holds that
\begin{align}
W_{N_1-4}^2(t')&\ls W_{N_1-4}^2(0)
+\int_0^{t'}W^2_{N_1-4}(t)\big\{M\delta+M\ve\w{t}^{-\frac98}\big\}dt,\label{curl:L2:low}\\
W_{N-1}^2(t')&\ls W_{N-1}^2(0)
+\int_0^{t'}W^2_{N-1}(t)\big\{M\delta+M\ve\w{t}^{-1}\big\}dt \nonumber\\
&\quad+\int_0^{t'}M^2\ve\delta\w{t}^{M'\ve-1}W_{N-1}(t)dt.\label{curl:L2:high}
\end{align}
\end{lemma}
\begin{proof}
Multiplying \eqref{curl:eqn:high} by $2\w{|x|}^2e^{q(|x|-t)}\Gamma^aw$ infers
\begin{equation*}
\begin{split}
&\quad\p_t\Big(e^q\big|\w{|x|}\Gamma^aw\big|^2\Big)
+\frac{e^q\big|\w{|x|}\Gamma^aw\big|^2}{\w{|x|-t}^\frac32}
+\dive\Big(e^qu\big|\w{|x|}\Gamma^aw\big|^2\Big)\\
&=e^q\big|\w{|x|}\Gamma^aw\big|^2\Big(\dive u+u\cdot\nabla q\Big)
+u\cdot\nabla\Big(\w{|x|}^2\Big)e^q|\Gamma^aw|^2
+\sum_{\substack{b+c=a,\\c<a}}2e^q\w{|x|}^2C^a_{bc}\Gamma^awJ_{bc}.
\end{split}
\end{equation*}
Integrating the above equality over $[0,t']\times\R^2$ leads to
\begin{equation}\label{curl:L2:1}
\begin{split}
&\quad W^2_{|a|}(t')+\int_0^{t'}\int
\frac{\big|\w{|x|}\Gamma^aw\big|^2}{\w{|x|-t}^\frac32}dxdt\\
&\ls W^2_{|a|}(0)+\int_0^{t'}\int\big|\w{|x|}\Gamma^aw\big|^2\cJ(u)dxdt
+\sum_{\substack{b+c=a,\\c<a}}\int_0^{t'}\int\w{|x|}^2|\Gamma^aw||J_{bc}|dxdt,
\end{split}
\end{equation}
where
\begin{equation}\label{cJ:def}
\cJ(u):=|\dive u|+\frac{|u|}{\w{|x|-t}^\frac32}+\frac{|u|}{\w{|x|}}.
\end{equation}
The integral domain is still decomposed into two parts of $|x|\ge\w{t}/8$ and $|x|\le\w{t}/8$ as before.

\noindent\underline{The case $|b|\le N_1-4$ in the region $|x|\le\w{t}/8$:}
By plugging the improved pointwise estimates \eqref{velocity:pw} into $\cJ(u)$ and $\tilde\Gamma^bu$, we derive
\begin{equation}\label{curl:L2:2}
\begin{split}
&\int_{|x|\le\w{t}/8}\big|\w{|x|}\Gamma^aw\big|^2\cJ(u)dx
+\sum_{\substack{b+c=a,\\c<a,|b|\le N_1-4}}\int_{|x|\le\w{t}/8}
\w{|x|}^2|\Gamma^aw||J_{bc}|dx\\
&\ls W^2_{|a|}(t)\big\{M\delta+M\ve\w{t}^{-\frac98}\big\}.
\end{split}
\end{equation}
\noindent\underline{The case $|b|\le N_1-4$ in the region $|x|\ge\w{t}/8$:}
Note that $|c|\le|a|-1\le N-2$.
Then, from \eqref{pw:wave3} and \eqref{grad:curl:cone}, we arrive at
\begin{equation}\label{curl:L2:3}
\begin{split}
&\int_0^{t'}\int_{|x|\ge\w{t}/8}\big|\w{|x|}\Gamma^aw\big|^2\cJ(u)dxdt
+\sum_{\substack{b+c=a,\\c<a,|b|\le N_1-4}}\int_0^{t'}\int_{|x|\ge\w{t}/8}
\w{|x|}^2|\Gamma^aw||J_{bc}|dxdt\\
&\ls M\ve\sum_{|a'|\le|a|}\int_0^{t'}\int
\frac{\big|\w{|x|}\Gamma^{a'}w\big|^2}{\w{|x|-t}^\frac32}dxdt
+\int_0^{t'}W^2_{|a|}(t)\big\{M\delta+M\ve\w{t}^{M'\ve-\frac32}\big\}dt,
\end{split}
\end{equation}
where we have used Young's inequality.

For all $|a|\le N_1-4$, substituting \eqref{curl:L2:2} and \eqref{curl:L2:3} into \eqref{curl:L2:1} yields \eqref{curl:L2:low}.

In the rest part, we will focus on the case $|b|\ge N_1-3$.
From $|b|+|c|\le N-1\le2N_1-10$, one has $|c|\le N_1-7\le N-4$.

\noindent\underline{In the region $|x|\ge\w{t}/8$:}
Applying \eqref{pw:curl} and \eqref{grad:curl:cone} to $\nabla\Gamma^cw$ infers
\begin{equation}\label{curl:L2:4}
\begin{split}
\int_{|x|\ge\w{t}/8}\w{|x|}^2|\Gamma^aw||J_{bc}|dx
&\ls\w{t}^{-\frac32}W_{|a|}(t)E_{|b|}(t)W_{|c|+3}(t)\\
&\ls M\ve\w{t}^{M'\ve-\frac32}W_{|a|}(t)W_{N-1}(t).
\end{split}
\end{equation}
\underline{In the region $|x|\le\w{t}/8$:}
We deduce from \eqref{pw:curl}, \eqref{H1norm}, \eqref{grad:curl'}, \eqref{curl:equiv}
and the Hardy inequality that
\begin{equation}\label{curl:L2:5}
\begin{split}
\|\w{|x|}J_{bc}\|_{L^2(|x|\le\frac{\w{t}}{8})}
&\ls\|\w{|x|}^\frac52\nabla\Gamma^cw\|_{L^\infty(|x|\le\frac{\w{t}}{8})}
\Big\|\frac{\chi\tilde\Gamma^bu}{\w{|x|}^\frac32}\Big\|_{L_x^2}\\
&\ls\|\w{|x|}^\frac32\Gamma^{\le1}\Gamma^cw\|_{L^\infty}
\Big\|\frac{\chi\tilde\Gamma^bu}{|x|\log|x|}\Big\|_{L_x^2}\\
&\ls\w{t}^{-1}W_{|c|+3}(t)\big\{E_{|b|+1}(t)+\cW_{|b|}(t)\big\}\\
&\ls\w{t}^{-1}W_{N_1-4}(t)\big\{E_N(t)+W_{N-1}(t)\big\}\\
&\ls M\delta W_{N-1}(t)+M^2\ve\delta\w{t}^{M'\ve-1}.
\end{split}
\end{equation}
For all $|a|\le N-1$, by plugging  \eqref{curl:L2:2}--\eqref{curl:L2:5} into \eqref{curl:L2:1}, we obtain \eqref{curl:L2:high}.
\end{proof}

\begin{lemma}[$L^p$ estimates]\label{lem:curl:Lp}
Under bootstrap assumptions \eqref{bootstrap}, set $\ds\sW_{p,m}(t):=\sum_{|a|\le m}\|\w{|x|}^{\theta(p)}\Gamma^aw(t,x)\|_{L_x^p}$, where
\begin{equation*}
\begin{split}
\theta(p):=\left\{
\begin{aligned}
&1,&&p=\frac{10}{9},\frac{10}{7},\\
&8,&&p=5.
\end{aligned}
\right.
\end{split}
\end{equation*}
Then, for $p=\frac{10}{9},\frac{10}{7},5$, it holds that
\begin{align}
\sW^p_{p,N_1-4}(t')&\ls\sW^p_{p,N_1-4}(0)+\int_0^{t'}\sW^p_{p,N_1-4}(t)
\big\{M\delta+M\ve\w{t}^{-\frac98}\big\}dt,\label{curl:Lp:low}\\
\sW^p_{p,N_1}(t')&\ls\sW^p_{p,N_1}(0)+\int_0^{t'}\sW^p_{p,N_1}(t)
\big\{M\delta+M\ve\w{t}^{-1}\big\}dt\nonumber\\
&\quad+\int_0^{t'}M\delta\w{t}^{M'\ve}\big\{M\delta+M\ve\w{t}^{-1}\big\}
\sW^{p-1}_{p,N_1}(t)dt.\label{curl:Lp:high}
\end{align}
\end{lemma}
\begin{proof}
Multiplying \eqref{curl:eqn:high} by $p\w{|x|}^{p\theta(p)}e^{q(|x|-t)}|\Gamma^aw|^{p-2}\Gamma^aw$ shows that
\begin{equation*}
\begin{split}
&\p_t\Big(e^q\big|\w{|x|}^{\theta(p)}\Gamma^aw\big|^p\Big)
+\frac{e^q\big|\w{|x|}^{\theta(p)}\Gamma^aw\big|^p}{\w{|x|-t}^\frac32}
+\dive\Big(e^qu\big|\w{|x|}^{\theta(p)}\Gamma^aw\big|^p\Big)\\
&=e^q\big|\w{|x|}^{\theta(p)}\Gamma^aw\big|^p\Big(\dive u+u\cdot\nabla q\Big)
+u\cdot\nabla\Big(\w{|x|}^{p\theta(p)}\Big)e^q|\Gamma^aw|^p\\
&\quad+\sum_{\substack{b+c=a,\\c<a}}
pe^q\w{|x|}^{p\theta(p)}|\Gamma^aw|^{p-2}C^a_{bc}\Gamma^awJ_{bc}.
\end{split}
\end{equation*}
Integrating the above equality over $[0,t']\times\R^2$ leads to
\begin{equation}\label{curl:Lp:1}
\begin{split}
&\quad\|\w{|x|}^{\theta(p)}\Gamma^aw(t',x)\|^p_{L^p}+\int_0^{t'}\int
\frac{\big|\w{|x|}^{\theta(p)}\Gamma^aw\big|^p}{\w{|x|-t}^\frac32}dxdt\\
&\ls\|\w{|x|}^{\theta(p)}\Gamma^aw(0,x)\|^p_{L^p}
+\int_0^{t'}\int\big|\w{|x|}^{\theta(p)}\Gamma^aw\big|^p\cJ(u)dxdt\\
&\quad+\sum_{\substack{b+c=a,\\c<a}}\int_0^{t'}
\int\w{|x|}^{p{\theta(p)}}|\Gamma^aw|^{p-1}|J_{bc}|dxdt,
\end{split}
\end{equation}
see \eqref{cJ:def} for the definition of $\cJ(u)$.

Next, we deal with the two integrals on the right hand side of \eqref{curl:Lp:1} in the two different parts $|x|\ge\w{t}/8$ and $|x|\le\w{t}/8$, separately.

In the region $|x|\ge\w{t}/8$, with the help of \eqref{grad:curl:cone}, applying \eqref{pw:wave3'} to $\tilde\Gamma^bu$ yields that
\begin{equation}\label{curl:Lp:2}
\begin{split}
&\int_{|x|\ge\w{t}/8}\big|\w{|x|}^{\theta(p)}\Gamma^aw\big|^p\cJ(u)dx
+\sum_{\substack{b+c=a,\\c<a}}\int_{|x|\ge\w{t}/8}
\w{|x|}^{p{\theta(p)}}|\Gamma^aw|^{p-1}|J_{bc}|dx\\
&\ls M\ve\sum_{|a'|\le|a|}\int_0^{t'}\int
\frac{\big|\w{|x|}^{\theta(p)}\Gamma^{a'}w\big|^p}{\w{|x|-t}^\frac32}dxdt
+\int_0^{t'}\sW^p_{p,|a|}(t)\big\{M\delta+M\ve\w{t}^{M'\ve-\frac32}\big\}dt.
\end{split}
\end{equation}
For $|b|\le N_1-4$, in the region $|x|\le\w{t}/8$, similarly to \eqref{curl:L2:2}, we derive
\begin{equation}\label{curl:Lp:3}
\begin{split}
&\int_{|x|\le\w{t}/8}\big|\w{|x|}^{\theta(p)}\Gamma^aw\big|^p\cJ(u)dx
+\sum_{\substack{b+c=a,\\c<a,|b|\le N_1-4}}\int_{|x|\le\w{t}/8}
\w{|x|}^{p{\theta(p)}}|\Gamma^aw|^{p-1}|J_{bc}|dx\\
&\ls\sW^p_{p,|a|}(t)\big\{M\delta+M\ve\w{t}^{-\frac98}\big\}.
\end{split}
\end{equation}
For all $|a|\le N_1-4$, plugging \eqref{curl:Lp:2} and \eqref{curl:Lp:3} into \eqref{curl:Lp:1} implies \eqref{curl:Lp:low}.

Finally, we focus on the proof of \eqref{curl:Lp:high}. In fact, it only suffices to treat the case of $N_1-3\le|b|\le N_1$ in the region $|x|\le\w{t}/8$. At this time, $|c|\le N_1-|b|\le3\le N_1-7$ holds.

For $p=\frac{10}{9},\frac{10}{7}$, by the H\"{o}lder inequality, we have
\begin{equation}\label{curl:Lp:4}
\begin{split}
\|\w{|x|}^3\nabla\Gamma^cw\|_{L^p}
&\ls\|\w{|x|}^8\nabla\Gamma^cw\|_{L^5}\|\w{|x|}^{-5}\|_{L^\frac{5p}{5-p}}\\
&\ls\sW_{|c|+1}(t)\ls\sW_{N_1-4}(t)\ls M\delta.
\end{split}
\end{equation}
On the other hand, we can deduce from $N_1+2\le N$, the standard Sobolev embedding and Hardy inequality that
\begin{equation}\label{curl:Lp:5}
\begin{split}
\|\w{|x|}^{-\frac54}\tilde\Gamma^bu\|_{L_x^\infty}
&\ls\|\w{|x|}^{-\frac54}\tilde\Gamma^bu\|_{L_x^2}
+\|\w{|x|}^{-\frac54}\nabla\nabla^{\le1}\tilde\Gamma^bu\|_{L_x^2}\\
&\ls\Big\|\frac{\chi\tilde\Gamma^bu}{|x|\log|x|}\Big\|_{L_x^2}
+\w{t}^{-1}E_{|b|+2}(t)+\cW_{|b|+1}(t)\\
&\ls\w{t}^{-1}E_{|b|+2}(t)+\cW_{|b|+1}(t)\\
&\ls\w{t}^{M'\ve}\big\{M\delta+M\ve\w{t}^{-1}\big\},
\end{split}
\end{equation}
where we have used \eqref{div:curl:ineq} and \eqref{H1norm}.

Collecting \eqref{curl:Lp:4} and \eqref{curl:Lp:5} infers
\begin{equation}\label{curl:Lp:6}
\int_{|x|\le\w{t}/8}\w{|x|}^p|\Gamma^aw|^{p-1}|J_{bc}|dx
\ls\sW^{p-1}_{p,|a|}(t)M\delta\w{t}^{M'\ve}\big\{M\delta+M\ve\w{t}^{-1}\big\}.
\end{equation}
Now, we deal with the case $p=5$.
It is concluded from \eqref{pw:prepare1}, \eqref{grad:curl'} and \eqref{curl:Lp:5} that
\begin{equation}\label{curl:Lp:7}
\begin{split}
\|\w{|x|}^8J_{bc}\|^5_{L^5(|x|\le\w{t}/8)}
&\ls\Big(\sW_{|c|+3}(t)+M\delta\Big)^5\int\frac{|\chi\tilde\Gamma^bu|^5}{\w{|x|}^6}dx\\
&\ls\Big(\sW_{N_1-4}(t)+M\delta\Big)^5
\|\w{|x|}^{-\frac54}\tilde\Gamma^bu\|^3_{L_x^\infty}
\Big\|\frac{\chi\tilde\Gamma^bu}{|x|\log|x|}\Big\|^2_{L_x^2}\\
&\ls\w{t}^{5M'\ve}(M\delta)^5\Big(M\delta+M\ve\w{t}^{-1}\Big)^5.
\end{split}
\end{equation}
Consequently, substituting \eqref{curl:Lp:2}, \eqref{curl:Lp:3}, \eqref{curl:Lp:6} and \eqref{curl:Lp:7} into \eqref{curl:Lp:1}
yields \eqref{curl:Lp:high}.
\end{proof}

\begin{lemma}
Under bootstrap assumptions \eqref{bootstrap}, we have
\begin{equation}\label{curl:Lp:equiv}
\bW_{N_1-4}(t)\ls\sum_{p=\frac{10}{9},\frac{10}{7},5}\sW_{p,N_1-4}(t),
\qquad\bW_{N_1}(t)\ls\sum_{p=\frac{10}{9},\frac{10}{7},5}\sW_{p,N_1}(t).
\end{equation}
\end{lemma}
\begin{proof}
The proof is the same as Lemma \ref{lem:curl:equiv}, we omit the details here.
\end{proof}

\section{Proof of Theorem \ref{thm:2dChaplygin}}\label{sect6}

Before starting the proof Theorem~\ref{thm:2dChaplygin}, we do some preparations.
\begin{lemma}[Standard Gronwall's inequality]\label{lem:gronwall}
Let $A(t)$ be an amplitude function which verifying that $A(0)=0, \p_tA(t)\ge0$.
If
\begin{equation*}
h(t)\le C\big\{\cR+\int_0^th(s)\p_tA(s)ds\big\},
\end{equation*}
holds for some positive constants $C$ and $\cR$, then
\begin{equation*}
\begin{split}
h(t)\le C\cR(1+e^{CA(t)}).
\end{split}
\end{equation*}
\end{lemma}
\begin{lemma}[Gronwall's inequality]\label{lem:gronwall'}
For $t\delta\le\kappa$ and two positive constants $c$ and $C\ge1$, if
\begin{equation*}
\begin{split}
h(t)\le C\Big\{h(0)+cM\delta(1+t)^{M'\ve}
+\int_0^th(s)(M\delta+M\ve(1+s)^{-1})ds\Big\}
\end{split}
\end{equation*}
holds, then for $M'\ge3CM$ and $\kappa\le\frac{1}{2CM}$, we have
\begin{equation*}
\begin{split}
h(t)\le C(1+t)^{M'\ve}(1+e)(h(0)+cM\delta).
\end{split}
\end{equation*}
\end{lemma}
\begin{proof}
Let $H(t):=\int_0^th(s)(M\delta+M\ve(1+s)^{-1})ds$ and $A(t):=M\delta t+M\ve\ln(1+t)$,
then we find $\p_tA(t)=M\delta+M\ve(1+t)^{-1}\ge0$.
Thus, we get
\begin{equation*}
\p_tH(t)=h(t)\p_tA(t)
\le C\p_tA(t)\Big\{h(0)+cM\delta(1+t)^{M'\ve}+H(t)\Big\},
\end{equation*}
which implies
\begin{equation}\label{gronwall:1}
\p_t(e^{-CA(t)}H(t))
\le C\p_tA(t)e^{-CA(t)}\Big\{h(0)+cM\delta(1+t)^{M'\ve}\Big\}.
\end{equation}
Note that
\begin{equation}\label{gronwall:2}
\begin{split}
&~C\int_0^t\p_tA(s)e^{-CA(s)}(1+s)^{M'\ve}ds\\
=&~C\int_0^t\Big(M\delta+M\ve(1+s)^{-1}\Big)e^{-CM\delta s}(1+s)^{M'\ve-CM\ve}ds\\
\le&~CM\delta t(1+t)^{M'\ve-CM\ve}+CM\ve\int_0^t(1+s)^{M'\ve-CM\ve-1}ds\\
\le&~(1+t)^{M'\ve-CM\ve}\Big\{CM\kappa+\frac{CM}{M'-CM}\Big\}\\
\le&~(1+t)^{M'\ve-CM\ve}.
\end{split}
\end{equation}
By integrating \eqref{gronwall:1} and plugging \eqref{gronwall:2} into the resulted inequality, we arrive at
\begin{equation*}
\begin{split}
H(t)&\le e^{CM\delta t}(1+t)^{M\ve}\Big\{h(0)+cM\delta(1+t)^{M'\ve-M\ve}\Big\}\\
&\le e(1+t)^{M'\ve}(h(0)+cM\delta).
\end{split}
\end{equation*}
This completes the proof of Lemma \ref{lem:gronwall'}.
\end{proof}

Next, we begin to prove the main theorem of this paper.
\begin{proof}[Proof of Theorem \ref{thm:2dChaplygin}]
Let $\tilde E_m(t):=\sup_{0\le s\le t}E_m(s)$.
Analogously, we can also define $\tilde W_m(t)$ and $\tilde\sW_{p,m}(t)$.
But for convenience, we still denote $\tilde E_m(t),\tilde W_m(t),\tilde\sW_{p,m}(t)$ as $E_m(t),W_m(t),\sW_{p,m}(t)$.
By these notations, \eqref{energy:wave:high}, \eqref{energy:wave:low}, \eqref{curl:L2:low}, \eqref{curl:L2:high}, \eqref{curl:Lp:low} and \eqref{curl:Lp:high} with assumption $2M'\ve\le\frac18$ can be reduced to
\begin{equation*}
\begin{split}
E_N(t')&\le C_1\Big\{E_N(0)+M^2\ve\delta^\frac34\kappa^\frac14
+M^2\delta^\frac78\kappa^\frac98
+\int_0^{t'}E_N(t)\big\{M\delta+M\ve(1+t)^{-1}\big\}dt\Big\},\\
E_{N_1-4}(t')&\le C_1\Big\{E_{N_1-4}(0)+M^2\ve\delta^\frac34\kappa^\frac14
+M^2\delta^\frac78\kappa^\frac98
+\int_0^{t'}E_{N_1-4}(t)\big\{M\delta+M\ve(1+t)^{-\frac98}\big\}dt\Big\},\\
W_{N_1-4}(t')&\le C_1\Big\{W_{N_1-4}(0)
+\int_0^{t'}W_{N_1-4}(t)\big\{M\delta+M\ve(1+t)^{-\frac98}\big\}dt\Big\},\\
W_{N-1}(t')&\le C_1\Big\{W_{N-1}(0)+\frac{M}{M'}M\delta(1+t)^{M'\ve}
+\int_0^{t'}W_{N-1}(t)\big\{M\delta+M\ve(1+t)^{-1}\big\}dt\Big\},\\
\sW_{p,N_1-4}(t')&\le C_1\Big\{\sW_{p,N_1-4}(0)
+\int_0^{t'}\sW_{p,N_1-4}(t)\big\{M\delta+M\ve(1+t)^{-\frac98}\big\}dt\Big\},\\
\sW_{p,N_1}(t')&\le C_1\Big\{\sW_{p,N_1}(0)+(\kappa+\frac{M}{M'})M\delta(1+t)^{M'\ve}
+\int_0^{t'}\sW_{p,N_1}(t)\big\{M\delta+M\ve(1+t)^{-1}\big\}dt\Big\},
\end{split}
\end{equation*}
where the positive constant $C_1>1$ is assumed to be suitably large.

Applying the Gronwall inequalities in Lemma \ref{lem:gronwall} and \ref{lem:gronwall'} to the above
equalities with amplitude function $A(t)=M\delta t+M\ve\ln(1+t)$ and $A(t)=M\delta t+8M\ve(1-(1+t)^{-\frac18})$,
respectively, we then obtain
\begin{equation*}
\begin{split}
E_N(t)&\le C_1(1+e^{C_1M\kappa}(1+t)^{C_1M\ve})
(E_N(0)+M^2\ve\delta^\frac34\kappa^\frac14+M^2\ve\kappa^\frac98),\\
E_{N_1-4}(t)&\le C_1(1+e^{C_1M(\kappa+8\ve)})
(E_{N_1-4}(0)+M^2\ve\delta^\frac34\kappa^\frac14+M^2\ve\kappa^\frac98),\\
W_{N_1-4}(t)&\le C_1W_{N_1-4}(0)(1+e^{C_1M(\kappa+8\ve)}),\\
W_{N-1}(t)&\le C_1(1+t)^{M'\ve}(1+e)\Big\{W_{N-1}(0)+\frac{M}{M'}M\delta\Big\},\\
\sW_{p,N_1-4}(t)&\le C_1\sW_{p,N_1-4}(0)(1+e^{C_1M(\kappa+8\ve)}),\\
\sW_{p,N_1}(t)&\le C_1(1+t)^{M'\ve}(1+e)
\Big\{\sW_{p,N_1}(0)+(\kappa+\frac{M}{M'})M\delta\Big\},
\end{split}
\end{equation*}
where $\kappa\le\frac{1}{2C_1M}$ and $M'\ge3C_1M$.
If $\ve_0\le\frac{1}{C_1M}$, then it concludes from the above inequalities with \eqref{initial:data}, \eqref{delta:def}, \eqref{H1norm}, \eqref{curl:equiv} and \eqref{curl:Lp:equiv} that there exists positive constants $C_2>1$ and $C_3>0$ such that
\begin{equation*}
\begin{split}
E_N(t)+\cX_N(t)&\le C_1C_2\w{t}^{M'\ve}
(C_3\ve+M\ve(M\delta)^\frac34+M\ve\kappa^\frac18),\\
E_{N_1-4}(t)+\cX_{N_1-4}(t)&\le C_1C_2
(C_3\ve+M\ve(M\delta)^\frac34+M\ve\kappa^\frac18),\\
W_{N-1}(t)+\cW_{N-1}(t)+\bW_{N_1}(t)+\sW_{N_1}(t)&\le C_1C_2\w{t}^{M'\ve}(C_3\delta+(\kappa+\frac{M}{M'})M\delta),\\
W_{N_1-4}(t)+\cW_{N_1-4}(t)+\bW_{N_1-4}(t)+\sW_{N_1-4}(t)&\le C_1C_2C_3\delta.
\end{split}
\end{equation*}
Choosing $M=4C_1C_2C_3$, $M'=4C_1C_2M$, $\kappa=\min\{\frac{1}{(4C_1C_2)^8},\frac{1}{2C_1M}\}$, $\ve_0=\frac{1}{16M'}$\footnote{More precisely, $\ve_0$ should be the minor one of $\frac{1}{16M'}$ and the smallness of $M\ve_0$ which has been used in the previous parts of this paper.}, $\delta_0=\min\{\frac{1}{(4C_1C_2)^\frac43M},\ve_0^\frac87\}$, we eventually achieve that for $t\delta\le\kappa$,
\begin{equation*}
\begin{split}
E_N(t)+\cX_N(t)&\le\w{t}^{M'\ve}(\frac14M\ve+\frac14M\ve+\frac14M\ve)
\le\frac34M\ve\w{t}^{M'\ve},\\
E_{N_1-4}(t)+\cX_{N_1-4}(t)&\le\frac14M\ve+\frac14M\ve+\frac14M\ve\le\frac34M\ve,\\
W_{N-1}(t)+\cW_{N-1}(t)+\bW_{N_1}(t)+\sW_{N_1}(t)&\le\w{t}^{M'\ve}
(\frac14\delta+\frac14\delta+\frac14\delta)\le\frac34M\delta\w{t}^{M'\ve},\\
W_{N_1-4}(t)+\cW_{N_1-4}(t)+\bW_{N_1-4}(t)+\sW_{N_1-4}(t)
&\le\frac14\delta\le\frac34M\delta.
\end{split}
\end{equation*}
This, together with the local existence of classical solution to \eqref{reducedEuler}, yields that \eqref{reducedEuler} admits a unique solution $(\sigma,u)\in C([0,\frac{\kappa}{\delta}],H^N(\R^2))$.
Thus, it completes the proof Theorem~\ref{thm:2dChaplygin}.
\end{proof}

\appendix
\setcounter{equation}{0}
\section{Derivation of the wave equation for the potential function $\phi$}\label{appendix:A}

\begin{lemma} We have
\begin{equation}\label{A1}
\dive(u\cdot\nabla u)=\frac12\Delta|u|^2-\curl(u\curl u).
\end{equation}
\end{lemma}
\begin{proof}
At first, it is obvious to see that
\begin{equation*}
\dive(u\cdot\nabla u)=u\cdot\nabla\dive u+\p_iu_j\p_ju_i.
\end{equation*}
Note that $\curl u=\eps_{ij}\p_iu_j$,
where the volume form $\eps_{ij}$ is the sign of the arrangement $\{ij\}$.
Then we find that
\begin{equation*}
\begin{split}
u\cdot\nabla\dive u
&=u_j\p_j\p_iu_i=u_j\p_i(\p_ju_i-\p_iu_j)+u_j\Delta u_j\\
&=u_j\p_i(\eps_{ji}\curl u)+u_j\Delta u_j\\
&=\p_i(u_j\eps_{ji}\curl u)-\eps_{ji}\p_iu_j\curl u+u_j\Delta u_j\\
&=-\curl(u\curl u)+(\curl u)^2+u_j\Delta u_j.
\end{split}
\end{equation*}
On the other hand, one has
\begin{equation*}
\frac12\Delta|u|^2=u_j\Delta u_j+\p_iu_j\p_iu_j.
\end{equation*}
Therefore, we achieve
\begin{equation*}
\begin{split}
\dive(u\cdot\nabla u)
&=-\curl(u\curl u)+(\curl u)^2+\frac12\Delta|u|^2+\p_iu_j(\p_ju_i-\p_iu_j)\\
&=-\curl(u\curl u)+(\curl u)^2+\frac12\Delta|u|^2+\p_iu_j\eps_{ji}\curl u\\
&=-\curl(u\curl u)+\frac12\Delta|u|^2.
\end{split}
\end{equation*}
\end{proof}
Next we derive the wave equation  \eqref{potential:wave} of the potential function $\phi$.
By taking divergence of the velocity equation in \eqref{reducedEuler}, we arrive at
\begin{equation*}
\begin{split}
0&=\p_t\dive u+\dive(u\cdot\nabla u)+\Delta(\sigma-\frac12\sigma^2)\\
&=\Delta(\p_t\phi+\frac12|u|^2+\sigma-\frac12\sigma^2)-\curl(u\curl u),
\end{split}
\end{equation*}
where we have used \eqref{A1}. Hence,
\begin{equation}\label{dt:phi}
\p_t\phi+\frac12|u|^2+\sigma-\frac12\sigma^2=\cA=-(-\Delta)^{-1}\curl(u\curl u),
\end{equation}
and then  \eqref{sigma:potential} is obtained.
Acting $\p_t$ on the both sides of \eqref{dt:phi} yields
\begin{equation}\label{dtt:phi}
\p_t^2\phi+\frac12\p_t(|u|^2-\sigma^2)+\p_t\sigma=\p_t\cA.
\end{equation}
On the other hand, from the first equation in \eqref{reducedEuler}, we get
\begin{equation*}
\p_t\sigma=-\Delta\phi+Q_1.
\end{equation*}
Substituting this into \eqref{dtt:phi} derives
\begin{align}\label{A4}
\Box\phi+\frac12\p_t(|u|^2-\sigma^2)+Q_1=\p_t\cA.
\end{align}
By using \eqref{reducedEuler} to $\p_t(|u|^2-\sigma^2)$, one has
\begin{align}\label{A5}
\frac12\p_t(|u|^2-\sigma^2)=u\cdot\p_tu-\sigma\p_t\sigma
=u\cdot Q_2-u\cdot\nabla\sigma-\sigma Q_1+\sigma\dive u
=u\cdot Q_2+(1-\sigma)Q_1.
\end{align}
Substituting \eqref{A5} into \eqref{A4} yields \eqref{potential:wave}.


\begin{thebibliography}{99}

\bibitem{Alinhac92} S. Alinhac, {\it Une solution approch\'ee en grand temps des \'equations d'Euler compressibles axisym\'etriques en dimension deux,} Comm. Partial Differential Equations \textbf{17} (1992), no. 3-4, 447--490.

\bibitem{Alinhac93} S. Alinhac, {\it Temps de vie des solutions r\'eguli\'eres des \'equations d'Euler compressibles axisym\'etriques en dimension deux,} Invent. Math. \textbf{111} (1993), 627--670.


\bibitem{Alinhac95} S. Alinhac, {\it   Blowup for nonlinear hyperbolic equations}. Progress in Nonlinear Differential Equations and their Applications, 17. Birkh\"auser Boston, Inc., Boston, MA, 1995.

\bibitem{Alinhac99} S. Alinhac, {\it Blowup of small data solutions for a class of quasilinear wave equations in two space dimensions. II}, Acta Math. \textbf{182} (1999), no. 1, 1--23.

\bibitem{Alinhac01} S. Alinhac, {\it The null condition for quasilinear wave equations in two space dimensions I,} Invent. Math. \textbf{145} (2001), no. 3, 597--618.

\bibitem{Alinhac10} S. Alinhac, {\it Geometric analysis of hyperbolic differential equations: an introduction.}
 London Mathematical Society Lecture Note Series, 374. Cambridge University Press, Cambridge, 2010. x+118 pp.



\bibitem{Christodoulou07} D. Christodoulou, {\it The formation of shocks in 3-dimensional fluids,} EMS Monogr. Math., Eur. Math. Soc., Z\"urich, 2007.

\bibitem{CM14} D. Christodoulou, Miao Shuang, {\it Compressible flow and Euler's equations}, Surveys of Modern Mathematics, \textbf{9}, International Press, Somerville, MA; Higher Education Press, Beijing, 2014.



\bibitem{CF} R. Courant, K. O. Friedrichs, \textit{Supersonic flow and shock waves}. Interscience Publishers Inc., New York, 1948.


\bibitem{DWY15}  Ding Bingbing, Ingo Witt,  Yin Huicheng, {\it The global smooth symmetric solution to 2-D full compressible Euler system of Chaplygin gases,} J. Differential Equations \textbf{258} (2015), no. 2, 445--482.



\bibitem{Godin05} P. Godin, {\it The lifespan of a class of smooth spherically symmetric solutions of the compressible Euler equations with variable entropy in three space dimensions,} Arch. Ration. Mech. Anal. \textbf{177} (2005), no. 3, 479--511.

\bibitem{Godin07} P. Godin, {\it Global existence of a class of smooth 3D spherically symmetric flows of Chaplygin gases with variable entropy,} J. Math. Pures Appl. \textbf{87} (2007), 91--117.

\bibitem{GIP16} Guo Yan, A.D. Ionescu, B. Pausader, {\it Global solutions of the Euler-Maxwell two-fluid system in 3D,} Ann. of Math. (2) \textbf{183} (2016), 377--498.

\bibitem{HKSW16} G. Holzegel, S. Klainerman, J. Speck, W.W.-Y. Wong, {\it Small-data shock formation in solutions to 3d quasilinear wave equations: An overview,} Journal of Hyperbolic Differential Equations \textbf{13} (2016), no. 01, 1--105.

\bibitem{Hormander97book} L. H\"ormander, {\it Lectures on nonlinear hyperbolic differential equations.} Math\'ematiques \& Applications (Berlin) [Mathematics \& Applications], \textbf{26}. Springer-Verlag, Berlin, 1997.

\bibitem{HoshigaKubo04} A. Hoshiga, H. Kubo, {\it Global solvability for systems of nonlinear wave equations with multiple speeds in two space dimensions,} Differential Integral Equations \textbf{17} (2004), no. 5-6, 593--622.

\bibitem{HouYin19} Hou Fei, Yin Huicheng, {\it Global smooth axisymmetric solutions to 2D compressible Euler equations of Chaplygin gases with non-zero vorticity,} J. Differential Equations \textbf{267} (2019), no. 5, 3114--3161.

\bibitem{HouYin20} Hou Fei, Yin Huicheng, {\it On global axisymmetric solutions to 2D compressible full Euler equations of Chaplygin gases,} Discrete Contin. Dyn. Syst. \textbf{40} (2020) no. 3, 1435--1492.

\bibitem{HouYin20jde} Hou Fei, Yin Huicheng, {\it Global small data smooth solutions of 2-D null-form wave equations with non-compactly supported initial data,} J. Differential Equations \textbf{268} (2020), no. 2, 490--512.

\bibitem{IL18} A.D. Ionescu, V. Lie, {\it Long term regularity of the one-fluid Euler-Maxwell system in 3D with vorticity,} Adv. Math. \textbf{325} (2018), 719--769.

\bibitem{John90} F. John, {\it Nonlinear wave equations, formation of singularities.}
    Seventh Annual Pitcher Lectures delivered at Lehigh University, Bethlehem, Pennsylvania, April 1989. University Lecture Series, \textbf{2}, American Mathematical Society, Providence, RI, 1990.





\bibitem{Lei16} Lei Zhen, {\it Global well-posedness of incompressible elastodynamics in two dimensions,} Comm. Pure Appl. Math. \textbf{69} (2016), no. 11, 2072--2106.


\bibitem{LukSpeck18} J. Luk, J. Speck, {\it Shock formation in solutions to the 2D compressible Euler equations in the presence of non-zero vorticity,} Invent. Math. \textbf{214} (2018), no. 1, 1--169.

\bibitem{Majda84book} A. Majda, {\it Compressible fluid flow and systems of conservation laws in several space variables,} Applied Mathematical Sciences, \textbf{53}, Springer-Verlag, New York, 1984.


\bibitem{Rammaha89} M.A. Rammaha, {\it Formation of singularities in compressible fluids in two-space dimensions,} Proc. Am. Math. Soc. \textbf{107} (1989), 705--714.


\bibitem{Sideris85} T.C. Sideris, {\it Formation of singularities in three-dimensional compressible fluids,} Comm. Math. Phys. \textbf{101} (1985), 475--485.

\bibitem{Sideris97} T.C. Sideris, {\it Delayed singularity formation in 2D compressible flow,} Amer. J. Math. \textbf{119} (1997), 371--422.


%

\bibitem{Stein} E.M. Stein, {\it Harmonic Analysis: Real-Variable Methods, Orthogonality, and Oscillatory Integrals,} Princeton University Press, Princeton, 1993.

\bibitem{Yin} Yin Huicheng, {\it Formation and construction of a shock wave for 3-D compressible Euler equations with the spherical initial data,} Nagoya Math. J., \textbf{175} (2004), 125--164.

\end{thebibliography}
\end{document}